\documentclass[11pt]{amsart}

\usepackage{amsmath, amssymb, amsthm,mathtools,tikz,tikz-cd,graphicx,color,stmaryrd,hyperref,enumerate,mathdots,mathrsfs,dynkin-diagrams}
\usepackage[new]{old-arrows}
\usepackage{colonequals}
\usepackage[shortlabels]{enumitem}
\usepackage[noadjust]{cite}
\usepackage[utf8]{inputenc}

\usepackage{appendix}

\usepackage[margin=1in]{geometry}
\usepackage{cleveref}

\usepackage[english]{babel}

\hypersetup{
    colorlinks,
    citecolor=blue,
    filecolor=blue,
    linkcolor=blue,
    urlcolor=blue
}

\DeclareSymbolFont{extraup}{U}{zavm}{m}{n}
\DeclareMathSymbol{\varheart}{\mathalpha}{extraup}{86}

\newcommand{\End}{\operatorname{End}}

\newcommand{\sgn}{\operatorname{sgn}}

\newcommand{\Ind}{\operatorname{Ind}}

\newcommand{\CC}{{\mathbb C}}

\newcommand{\ZZ}{{\mathbb Z}}

\newcommand{\heckeinduced}{\Ind_{\mathcal H_k \otimes \mathcal H_{n-k}}^{\mathcal H_n} \CC_\mathrm{triv}}
\newcommand{\TLinduced}{\Ind_{\TL_k \otimes \TL_{n-k}}^{\TL_n} \CC_\mathrm{triv}}

\newcommand{\TL}{\operatorname{TL}}

\newcommand{\heart}{\mathbin{\heartsuit}}
\newcommand{\dia}{\mathbin{\diamondsuit}}
\newcommand{\flip}{\operatorname{flip}}
\newcommand{\mH}{\mathcal H}
\newcommand{\mN}{\mathcal N}
\newcommand{\mM}{\mathcal M}

\renewcommand{\sf}{\sigma^{\mathrm{flip}}}

\newcommand{\eps}{\epsilon}
\newcommand{\quantum}{\mathcal U_q(\mathfrak{sl}_2)}
\newcommand{\bU}{\mathbf U}
\newcommand{\id}{\operatorname{id}}
\newcommand{\im}{\operatorname{Im}}
\renewcommand{\phi}{\varphi}
\renewcommand{\hom}{\operatorname{Hom}}

\DeclareMathOperator{\Hom}{Hom}

\newtheorem{thm}{Theorem}[section]

\newtheorem{lem}[thm]{Lemma}

\newtheorem{cor}[thm]{Corollary}
\newtheorem{prop}[thm]{Proposition}

\newtheorem{prop-def}[thm]{Proposition-Definition}

\theoremstyle{definition}
\newtheorem{defn}[thm]{Definition}
\newtheorem{ex}[thm]{Example}

\theoremstyle{remark}
\newtheorem*{rem}{Remark}

\title[The Dual Canonical Basis via the TL Algebra]{The Dual Canonical Basis in the Spin Representation \\ via the Temperley-Lieb Algebra}

\makeatletter
\let\@wraptoccontribs\wraptoccontribs
\makeatother

\address{PRIMES-USA}
\email[R.~Chen]{rachelrxchen@gmail.com}
\author{Rachel Chen}
\begin{document}

\begin{abstract}
The spin representation $(\CC^2)^{\otimes n}$ has a dual canonical basis introduced by Lusztig that is important in many areas of algebra, geometry, and physics. Khovanov observed that a portion of the dual canonical basis can be viewed diagrammatically through the Temperley-Lieb algebra. We provide a simpler construction that we generalize to the entire dual canonical basis, and write explicit formulas to compute the dual canonical basis, and thus the canonical basis of the spherical module, as a byproduct. We reprove some of Khovanov's results using our new perspective. Furthermore, we use the Hecke algebra to reprove the fact that the canonical basis is indeed dual to the dual canonical basis, leading to similar results about the canonical basis in $\mM^*$ and $\mN^*$, the dual spaces to the spherical and aspherical modules, as a byproduct. Finally, we present an alternative axiomatic definition of the canonical basis using diagrams. 
%
\end{abstract}

\maketitle



\section{Introduction}
\subsection{Background, History, and Motivation}
The spin representation $(\CC^2)^{\otimes n}$ is an important object in various areas of mathematics and physics. Since it was first introduced by Cartan in 1913 \cite{cartan} and applied to physics by Dirac in 1928 \cite{dirac}, mathematicians and physicists have studied it extensively. The spin representation parametrizes configurations of $n$ electrons, each of which has spin-up or spin-down. 

If one considers $(\CC^m)^{\otimes n}$ as a $(\mH_n, \mathcal U_q(\mathfrak{sl}_m))$-bimodule, quantum Schur-Weyl duality states that if $m \geq n$, the algebra of endomorphisms of $(\CC^m)^{\otimes n}$ that commute with the action of $\mathcal U_q(\mathfrak{sl}_m)$, where $q \in \CC$ is not a root of unity, is precisely $\mH_n$, the Hecke algebra with $n$ generators, which has a natural presentation as an algebra of diagrams on a plane from $n$ vertices to $n$ vertices subject to certain quadratic relations. However, as we are concerned with the specific case of $m = 2$, the algebra of endomorphisms will be a quotient of $\mH_n$. This quotient is precisely the algebra generated by diagrams on the plane from $n$ vertices to $n$ vertices without crossings, with no additional relations, called the Temperley-Lieb algebra $\TL_n$ (see Figure~\ref{fig:exTL}). First introduced by Temperley and Lieb in 1971 \cite{tl}, the Temperley-Lieb algebra has applications in representation theory and knot theory, both due to its connection with the spin representation and quantum groups, and also independently. 

\begin{figure}[!ht]
\tikzset{every picture/.style={line width=0.75pt}} 

\begin{tikzpicture}[scale=.25]
    \node at (1,9){$\bullet$};
    \node at (5,9){$\bullet$};
    \node at (9,9){$\bullet$};
    \node at (1,0){$\bullet$};
    \node at (5,0){$\bullet$};
    \node at (9,0){$\bullet$};
\draw (0,0) -- (10,0);
\draw (0,9) -- (10,9);
\draw (1,0) arc (180:0:2);
\draw (9,9) arc (360:180:2);
\draw (1,9) .. controls (1,1) and (9,8) .. (9,0);
\end{tikzpicture}
\hspace{.5cm}
\begin{tikzpicture}[scale=.25]
    \node at (1,9){$\bullet$};
    \node at (5,9){$\bullet$};
    \node at (9,9){$\bullet$};
    \node at (13,9){$\bullet$};
    \node at (1,0){$\bullet$};
    \node at (5,0){$\bullet$};
    \node at (9,0){$\bullet$};
    \node at (13,0){$\bullet$};
\draw (0,0) -- (14,0);
\draw (0,9) -- (14,9);
\draw (13,0) arc (0:180:2);
\draw (1,9) arc (180:360:2);
\draw (1,0) .. controls (1,8) and (9,1) .. (9,9);
\draw (5,0) .. controls (5,8) and (13,1) .. (13,9);
\end{tikzpicture}
\hspace{.5cm}
\begin{tikzpicture}[scale=.25]
    \node at (1,9){$\bullet$};
    \node at (5,9){$\bullet$};
    \node at (9,9){$\bullet$};
    \node at (13,9){$\bullet$};
    \node at (1,0){$\bullet$};
    \node at (5,0){$\bullet$};
    \node at (9,0){$\bullet$};
    \node at (13,0){$\bullet$};
\draw (0,0) -- (14,0);
\draw (0,9) -- (14,9);
\draw (5,0) arc (180:0:2);
\draw (1,9) arc (180:360:2);
\draw (9,9) arc (180:360:2);
\begin{scope}
\clip (0,0) rectangle (14,9);
\draw (7,0) ellipse (6 and 4);
\end{scope}
\end{tikzpicture}
\caption{Examples of some diagrams in the Temperley-Lieb algebra. }
\label{fig:exTL}
\end{figure}
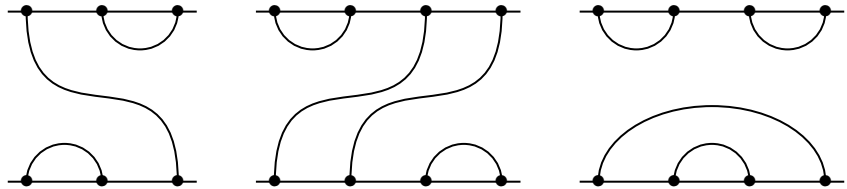

One breakthrough in the study of representations of quantum groups is Lusztig's discovery of the dual canonical basis, a basis satisfying certain upper triangularity and self-dual conditions in an arbitrary finite dimensional representation of a semisimple Lie algebra, in 1990 \cite{lusztigintro}. Lusztig's dual canonical basis quickly led to discoveries about its relation to the crystal basis \cite{kashiwara} and algebraic geometry \cite{lusztigintro2}. Since then, mathematicians have studied this dual canonical basis and related them with different areas of mathematics, including cluster algebras, total positivity, and categorification (see \cite{berenstein, qin, webster}). While leading to many breakthroughs, Lusztig's dual canonical basis was defined axiomatically and could not be explicitly computed. An important breakthrough is due to Khovanov \cite{khovanov}, who first related Lusztig's dual canonical basis and the Hecke algebra through the graphical calculus of diagrams, allowing mathematicians to understand the dual canonical basis explicitly. The understanding of this (dual) canonical basis also allows the explicit computation of characters of certain irreducible highest-weight $\mathfrak{sl}_n$-modules. Our paper expands on Khovanov's work. 

In the general case, the explicit computation of the dual canonical basis is extremely complicated. However, for the specific case of $\quantum$, describing the dual canonical basis in $(\CC^2)^{\otimes n}$ is sufficient, as every finite dimensional irreducible representation is a quotient of $(\CC^2)^{\otimes n}$, and the dual canonical basis of the quotient is a subset of the dual canonical basis. 

\subsection{Structure of the Paper}
In the first six sections, we cover preliminaries; our results start from Section~\ref{bijectionsection}. 
 
Khovanov inductively constructs the dual canonical basis, and we first examine the dual canonical basis through the lens of the Temperley-Lieb algebra in Section~\ref{bijectionsection}, asserting the isomorphism $\bigoplus_{0 \leq k \leq n} \Ind_{\TL_k \otimes \TL_{n-k}}^{\TL_n}\CC_\mathrm{triv} \cong (\CC^2)^{\otimes n}$ and proving that the diagrammatic basis on the left hand side maps to the dual canonical basis on the right hand side, which we can explicitly construct. We use our results to reprove Khovanov's inductive computation though a new diagrammatic perspective in Section~\ref{computationinductive}. As a byproduct of our results, in Section~\ref{computationexplicit}, we write an explicit, non-inductive formula for the dual canonical basis in Theorem~\ref{dualcomputations}, which also leads to an explicit formula for the canonical basis of the spherical module in Corollary~\ref{sphericalcor}. In his thesis, Khovanov poses the question of how his results are related to the results of the Temperley-Lieb algebra and the Hecke algebra in \cite{fg}, and our work answers that. 

Khovanov also describes the dual canonical basis of the $\quantum$-invariant subspace of $(\CC^2)^{\otimes n}$ in terms of the Temperley-Lieb algebra. In Section~\ref{cor1.10section}, we reprove his result in a simpler way using diagrams in Theorem~\ref{embed}, that can be generalized to the entire dual canonical basis, as we describe in the previous section. 

Besides the dual canonical basis, another important basis in $(\CC^2)^{\otimes n}$ is the canonical basis, which is dual to the dual canonical basis. We revisit this fact and prove this duality through the perspective of the Hecke algebra in Section~\ref{dualthroughhecke}, by introducing a pairing between the spherical and aspherical modules. Our approach allows us to work with the spin representation through the Hecke algebra, which is more convenient as the bar involution in the Hecke algebra is simpler than the involution in $(\CC^2)^{\otimes n}$. As a byproduct of our work in the Hecke algebra, we also obtain formulas for the computation of the canonical bases in $\mM^*$ and $\mN^*$, the dual spaces of the spherical and aspherical modules. Finally, in Section~\ref{canonicalsec}, we describe an alternative axiomatic definition of the canonical basis, through a diagrammatic perspective, in Theorem~\ref{canonicalthm}. 

\subsection{Future Work}
Our approach via the Temperley-Lieb algebra not only allows mathematicians to visually understand these bases, but also has applications beyond making the dual canonical basis more tangible to work with. If one studies the analog of the spin representation in the affine setting (i.e. replacing the symmetric group by its affine analog), there is no known algebraic definition of the canonical basis in it. Although the analog of the dual canonical basis can be defined geometrically as a basis of classes of irreducible objects in the heart of a certain exotic perverse t-structure \cite{bezrukavnikov, bezrukavnikov2}, this geometric definition is complicated and doesn't give insight into its explicit computation. However, one can show that the canonical basis is still related to the diagrammatic basis of the affine Temperley-Lieb algebra, the algebra generated by the diagrams on a cylinder, in a similar way as explained in the finite case in this paper. A deeper understanding of the relation of the Temperley-Lieb algebra in the finite setting may allow us to understand the relation in the affine setting better, importantly allowing us to explicitly compute the characters of $\widehat{\mathfrak{sl}_n}$-modules. This strategy of computing the characters of $\widehat{\mathfrak{sl}_n}$-modules proved quite effective for a specific case in Proposition 2.19 and Proposition 2.21 of \cite{BKK}.

\section{The Temperley-Lieb Algebra}
In this section, we introduce the Temperley-Lieb Algebra, the object that is central to our paper. Throughout this paper, let $q$ be a complex number that is not a root of unity. 
\begin{defn}
An $n$-\emph{diagram} is two parallel lines with $n$ vertices on both the top and bottom line, where the vertices are connected by edges so that the edges are between the parallel lines, the edges do not intersect each other, and each vertex is the endpoint of exactly one edge. If two $n$-diagrams are homotopic to each other, they are considered the same diagram. 
\end{defn}

We refer to the edges of a diagram by \emph{links}. Call a link 
\begin{enumerate}
\item[$\bullet$] \emph{straight} if it connects the $i$th vertex on the bottom line to the $i$th vertex on the top line, 
\item[$\bullet$] \emph{quasi-simple} if it connects two vertices on the same line, and
\item[$\bullet$] \emph{simple} if it connects two adjacent vertices on either line. 
\end{enumerate}

\begin{defn}[Diagram Algebra]
Let $\ell_n$ denote the set of all $n$-diagrams and let $\beta = -q-q^{-1}$ be a complex number. The \emph{diagram algebra}, denoted $\CC \ell_n(\beta)$, is the algebra over $\CC$ generated by the basis $\ell_n$, with multiplication defined as the concatenation of diagrams, with a factor of $\beta$ multiplied for each closed loop formed during concatenation. 
\end{defn}

Let $e_i$ refer to the diagram with a simple link connecting the $i$th and $(i+1)$th vertices on each line and with straight links everywhere else (see Figure~\ref{fig:genex}). 

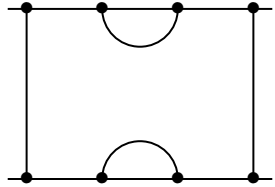
\begin{figure}[!ht]
\centering
\tikzset{every picture/.style={line width=0.75pt}} 

\begin{tikzpicture}[scale = .25]
    \node at (1,9){$\bullet$};
    \node at (5,9){$\bullet$};
    \node at (9,9){$\bullet$};
    \node at (13,9){$\bullet$};
    \node at (1,0){$\bullet$};
    \node at (5,0){$\bullet$};
    \node at (9,0){$\bullet$};
    \node at (13,0){$\bullet$};
\draw (0,0) -- (14,0);
\draw (0,9) -- (14,9);
\draw (9,0) arc (0:180:2);
\draw (5,9) arc (180:360:2);
\draw (1, 0) -- (1, 9);
\draw (13, 0) -- (13, 9);
\end{tikzpicture}

\caption{The diagram $e_2$ in $\CC \ell_4(\beta)$}
\label{fig:genex}
\end{figure}
\begin{lem}[\cite{degroot}]
The set $\{e_1, \dots, e_{n-1}\}$ generates $\CC \ell_n(\beta)$; in other words, concatenation of the diagrams $\{e_1, \dots, e_{n-1}\}$ yield all possible $n$-diagrams. 
\end{lem}
If we draw some diagrams, we can see that there are some relations between these generators. The following definition defines the Temperley-Lieb Algebra, the diagram algebra in terms of this basis instead of the diagrams. 
\begin{defn}[Temperley-Lieb Algebra]
The \emph{Temperley-Lieb Algebra}, denoted by $\TL_n(\beta)$ or just $\TL_n$, is the associative unital algebra over $\CC$ with words generated by the letters $e_1$, $e_2$, $\dots$, $e_{n-1}$ with relations 
\[
e_i^2 = \beta \cdot e_i, \; e_ie_{i \pm 1}e_i = e_i, \;\text{and }e_ie_j = e_je_i \text{ for } |i-j| \geq 2. 
\]
\end{defn}
\begin{thm}[\cite{degroot}]
There is an isomorphism of algebras between $\CC \ell_n(\beta)$ and $\TL_n(\beta)$ given by identifying $e_i$ with the corresponding diagram. 
\end{thm}

For the rest of this paper, we use $\TL_n$ interchangeably, referring to both the diagram algebra and the Temperley-Lieb algebra, and use $e_i$ to refer to both its corresponding diagram and the letter in the Temperley-Lieb algebra. As usual, call a word \emph{reduced} if it is of shortest possible length. 
%
%
\section{Combinatorical Bijections of the Temperley-Lieb Algebra} \label{tlcombo}
Diagrams are also related to other combinatorial structures. The goal of this section is to assert a bijection between the Tempeley-Lieb algebra and other combinatorial objects, laying the groundwork for Section~\ref{thm1.8}. All of the content in this section is from either \cite{degroot} or \cite{westbury}. 
We first define the parenthesis diagram. 
\begin{defn}
A $(n, p)$-\emph{parenthesis diagram} is an involution $\sigma \in S_n$ such that 
\begin{enumerate}
\item[(i)] there do not exist $i$, $j$, and $k$ with $i < j < k$, $\sigma(j) = j$, and $\sigma(i) = k$, 
\item[(ii)] there do not exist $i$, $j$, $k$, and $l$ with $i < j < k < l$, $\sigma(i) = k$, and $\sigma(j) = l$, and
\item[(iii)] there are $2p$ elements $i$ with $\sigma(i) \neq i$ and $n-2p$ elements $j$ such that $\sigma(j) = j$. 
\end{enumerate}
\end{defn}

Now, we define a link state. 
\begin{defn}
An $(n, p)$-\emph{link state} is half of a $n$-diagram, with exactly $p$ quasi-simple links. 
\end{defn}

In a link state, the links that are cut off are called \emph{defects}. If we just consider the quasi-simple links and ignore the defects, we obtain $(n, p)$-\emph{cup diagrams}. An example is shown in Figure~\ref{fig:cupdiagram}. 

\begin{figure}[!ht]
\centering
\tikzset{every picture/.style={line width=0.75pt}} 

\begin{tikzpicture}[scale = .25]
\node at (1,9){$\bullet$};
    \node at (5,9){$\bullet$};
    \node at (9,9){$\bullet$};
    \node at (13,9){$\bullet$};
\draw (0,9) -- (14,9);
\draw (5,9) arc (180:360:2);
\draw (1, 4.5) -- (1, 9);
\draw (13, 4.5) -- (13, 9);
\end{tikzpicture}
\hspace{10mm}
\begin{tikzpicture}[scale = .25]
\draw (1,9) circle[radius=4pt];
\fill (1,9) circle[radius=4pt];

\draw (5,9) circle[radius=4pt];
\fill (5,9) circle[radius=4pt];
\draw (5,9) arc (180:360:2);
\draw (9,9) circle[radius=4pt];
\fill (9,9) circle[radius=4pt];
\draw (13,9) circle[radius=4pt];
\fill (13,9) circle[radius=4pt];

\node at (1,9){$\bullet$};
    \node at (5,9){$\bullet$};
    \node at (9,9){$\bullet$};
    \node at (13,9){$\bullet$};
\end{tikzpicture}
\caption{An example of a $(4, 1)$-link state and its corresponding cup diagram. }
\label{fig:cupdiagram}
\end{figure}
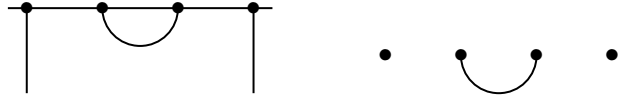
\begin{rem}
Note that the Temperley-Lieb algebra acts naturally on the vector space generated by all $(n, p)$-link states. This will be useful later in this paper. 
\end{rem}
\begin{lem} \label{cupparenthesis}
There is a bijection between $(n, p)$-parenthesis diagrams, $(n, p)$-link states, and $(n, p)$-cup diagrams. 
\end{lem}
%
%
%
\section{The Hecke Algebra and Kazhdan-Lusztig Basis}\label{heckesection}
\subsection{The Hecke Algebra}
The Hecke algebra, as described in \cite{losev}, is an important structure that is related to the Temperley-Lieb algebra and used throughout our paper. Recall that $q \in \CC$ is not a root of unity. 

Let $W$ be a Coxeter group with set $S$ of simple reflections. We have the Bruhat order $\prec$ on $W$, with $y \preceq x$ if and only if some substring of some reduced form of $x$ is a reduced form of $y$. Let $\ell$ denote the function returning the reduced length of a word. Note that for $s \in S$ and $w \in W$, $\ell(sw) = \ell(w) - 1$ if $w$ has a reduced expression that starts with $s$, and $\ell(sw) = \ell(w) + 1$ otherwise. Now, we can define the Hecke algebra. 

\begin{defn}[Hecke Algebra]
The Hecke algebra $\mathcal H = \mathcal H(W, S)$ is the free $\ZZ[q, q^{-1}]$-module with basis $H_w$, $w \in W$ and an associative product determined by 
\[ H_xH_y = H_{xy} \text{ if } \ell(xy) = \ell(x) + \ell(y)\]
\[ H_xH_s = H_{xs} + (q^{-1}-q)H_x \text{ if } \ell(xs) = \ell(x) -1. \]
\end{defn}
This basis $H_x$ is called the \emph{standard basis}. We can see that $H_1$ is a unit since $\ell(1) = 0$ and that the second relation can also be written as 
\[ H_sH_x = H_{sx} + (q^{-1}-q)H_x \text{ if } \ell(sx) = \ell(x) - 1. \]

\begin{rem}
Alternatively, the Hecke algebra can be defined by the generators $H_s$ where $s \in S$ and the following relations. The first relation 
\[ (H_s-q^{-1})(H_s + q) = 0\]
is called the quadratic relation and the other relations 
\[ H_sH_tH_s\cdots  = H_tH_sH_t\cdots  \text{ for } s \neq t \in S,\]
where the number of factors on both sides are equal to the order of $st$, are called the braid relations. 
\end{rem}
\subsection{The Kazhdan-Lusztig Basis of the Hecke Algebra}
The Hecke algebra has a ``canonical'' basis called the Kazhdan-Lusztig basis, defined using the bar involution and an upper triangularity property, that controls much of the representation theory of $\quantum$. 
\begin{lem}[Bar Involution \cite{kl}]
The map $a \mapsto \bar a : \mathcal H \rightarrow \mathcal H$ such that $\bar q = q^{-1}$ and $\bar H_x = (H_{x^{-1}})^{-1}$ is a ring involution. 
\end{lem}
From $H_s^2 = (q^{-1}-q)H_s + 1$, we obtain $H_s^{-1} = H_s + q - q^{-1}$. We call $a$ a \emph{self-dual} element if $\bar a = a$. 
\begin{lem}[\cite{kl}]\label{heckeselfdual}
For all $x \in W$, there exists a unique self-dual element $B_x \in \mathcal H$ such that $B_x \in H_x + \sum_y q^{-1}\ZZ[q^{-1}] H_y$. Furthermore, we have $B_x \in H_x + \sum_{y\prec x} q^{-1}\ZZ[q^{-1}] H_y$. 
\end{lem}
\begin{rem}
The first part of this lemma is important for its proof. The second part gives upper triangularity, which is a central part of the structure of the Kazhdan-Lusztig basis and the proof of the following theorem. 
\end{rem}
\begin{rem}\label{altrem}
Alternatively, we can take the ring involution fixing all $H_i$ and sending $q$ to $-q^{-1}$, so we can also define the Kazhdan-Lusztig basis as $B_x \in H_x + \sum_{y\prec x} q\ZZ[q] H_y$. In this paper, we use $q^{-1}\ZZ[q^{-1}]$ instead for the later compatibility with the Temperley-Lieb algebra. 
\end{rem}
\begin{thm}[Kazhdan-Lusztig Basis \cite{kl}] \label{klbasis}
The $B_x$ as defined above form a basis of $\mathcal H$. 
\end{thm}
Throughout this paper, the Kazhdan-Lusztig basis will also be called the canonical basis of the Hecke algebra. 

\begin{defn}[Kazhdan-Lusztig Polynomial] \label{klpoly}
For $x, y \in W$, the Kazhdan-Lusztig polynomial $p_{y, x} \in \ZZ[q^{-1}]$ is such that $B_x = \sum_y p_{y, x}H_y$. 
\end{defn}
Note that $p_{y, x} = 0$ for $y \npreceq x$, $p_{x, x} = 1$, and $p_{y, x} \in q^{-1}\ZZ[q^{-1}]$ for $y \prec x$. Furthermore, note that it is possible to explicitly compute the Kazhdan-Lusztig basis. 

This proposition demonstrates that the basis of the Temperley-Lieb algebra can be thought of as an analog to the Kazhdan-Lusztig basis of the Hecke Algebra. 

\section{Action of the Temperley-Lieb Algebra}\label{tlaction}
In this section, we discuss the spin representation $(\CC^2)^{\otimes n}$, an important object in many areas of mathematics and physics, and how the Temperley-Lieb algebra acts on it. To define the spin representation, we first define the Temperley-Lieb category. 

\subsection{The Temperley-Lieb Category} The Temperley-Lieb category gives a novel perspective of the Temperley-Lieb algebra, providing an easier way to visualize the spin representation. 

\begin{defn}
    A \emph{$(m, n)$-diagram} comprises two parallel lines with $m$ vertices on the bottom line and $n$ vertices on the top line such that the vertices are connected by edges satisfying the following properties: 
    \begin{itemize}
        \item the edges are between the parallel lines, 
        \item the edges do not intersect each other, and
        \item each vertex is the endpoint of exactly one edge. 
    \end{itemize}
\end{defn}

Given a $(\ell, m)$-diagram and a $(m, n)$-diagram, we can obtain a $(\ell, n)$-diagram by concatenation. Similarly to the case of the Temperley-Lieb algebra, in concatenation, a contractible loop may appear; we remove it and multiply the resulting diagram by formal variable $\beta = -q-q^{-1}$. An example of concatenation is shown in Figure~\ref{fig:comp}. We define a $(a, b)$-diagram composed with a $(c, d)$-diagram to be equal to $0$ if $b \neq c$. 

    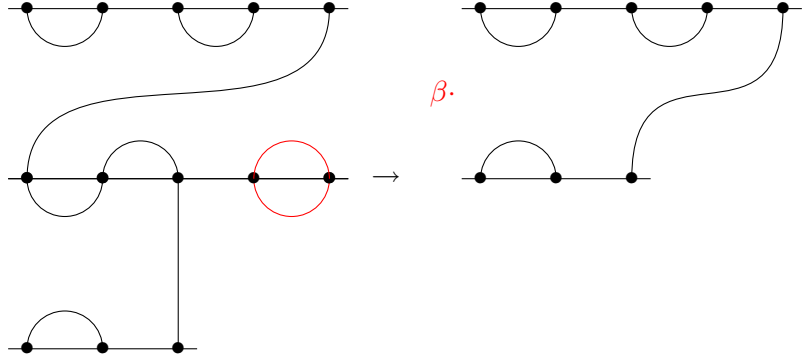
\begin{figure}[!ht]\label{fig:comp}
    \centering
    \vspace{5mm}
    \begin{tikzpicture}[scale=.25]
    \draw (0,0) -- (10,0);
    \draw (0,9) -- (18,9);
    \node at (1,9){$\bullet$};
    \node at (5,9){$\bullet$};
    \node at (17,9){$\bullet$};
    \node at (1,0){$\bullet$};
    \node at (5,0){$\bullet$};
    \node at (9,0){$\bullet$};
    \draw[red] (17,9) arc (360:180:2);
    \draw (9,9) -- (9,0);    
    \draw (0,9) -- (18,9);
    \draw (5,0) arc (0:180:2);
    \draw (1,9) arc (180:360:2);
    
\draw (0,18) -- (18,18);
\node at (1,18){$\bullet$};
\node at (5,18){$\bullet$};
\node at (9,18){$\bullet$};
\node at (13,18){$\bullet$};
\node at (17,18){$\bullet$};
\node at (1,9){$\bullet$};
\node at (5,9){$\bullet$};
\node at (9,9){$\bullet$};
\node at (13,9){$\bullet$};
\node at (17,9){$\bullet$};
\draw (5,9) arc (180:0:2);
\draw[red] (13,9) arc (180:0:2);
\draw (13,18) arc (360:180:2);
\draw (5,18) arc (360:180:2);
\draw (1,9) .. controls (1,17) and (17,10) .. (17,18); 
    \node at (20, 9){$\rightarrow$};
    
    \node[red] at (23, 13.5){$\beta \cdot$};
\draw (24,9) -- (34,9);
\draw (24,18) -- (42,18);
\node at (25,18){$\bullet$};
\node at (29,18){$\bullet$};
\node at (33,18){$\bullet$};
\node at (37,18){$\bullet$};
\node at (41,18){$\bullet$};
\node at (25,9){$\bullet$};
\node at (29,9){$\bullet$};
\node at (33,9){$\bullet$};
\draw (37,18) arc (360:180:2);
\draw (29,18) arc (360:180:2);
\draw (33,9) .. controls (33,17) and (41,10) .. (41,18); 
\draw (29,9) arc (0:180:2);
    
    \end{tikzpicture}
    \caption{The composition of a $(3, 5)$-diagram and a $(5, 5)$-diagram. }
    \label{fig:comp}
\end{figure}

\begin{defn}[Temperley-Lieb Category]
The \emph{Temperley-Lieb category}, denoted $\TL$, has objects $\{[n], n \in \ZZ_{\geq 0}\}$, where $[n]$ is a set of $n$ vertices. The morphisms $\hom([m], [n])$ is the vector space with a basis of diagrams consisting of $(m, n)$-diagrams. 
\end{defn}

Define $\eps_i^n \in \hom([n], [n-2])$ to be the diagram connecting the $i$th and $(i+1)$th vertices on the bottom line and containing no other quasi-simple links. Similarly, define $\delta_i^n \in \hom([n-2], [n])$ to be the diagram connecting the $i$th and $(i+1)$th vertices on the top line with no other quasi-simple links. An example of these diagrams are shown in Figure~\ref{fig:gen}. 
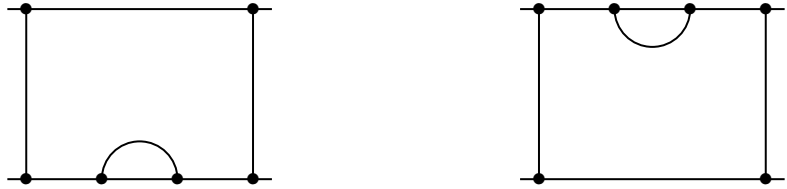
\begin{figure}[!ht]
\centering
\tikzset{every picture/.style={line width=0.75pt}} 
\begin{tikzpicture}[scale=.25]
    \node at (1,9){$\bullet$};
    \node at (13,9){$\bullet$};
    \node at (1,0){$\bullet$};
    \node at (5,0){$\bullet$};
    \node at (9,0){$\bullet$};
    \node at (13,0){$\bullet$};
\draw (0,0) -- (14,0);
\draw (0,9) -- (14,9);
\draw (1, 0) --(1, 9);
\draw (13, 0) --(13, 9);
\draw (5,0) arc (180:0:2);
\end{tikzpicture}
\hspace{30mm}
\begin{tikzpicture}[scale=.25]
    \node at (1,9){$\bullet$};
    \node at (5,9){$\bullet$};
    \node at (9,9){$\bullet$};
    \node at (13,9){$\bullet$};
    \node at (1,0){$\bullet$};
    \node at (13,0){$\bullet$};
\draw (0,0) -- (14,0);
\draw (0,9) -- (14,9);
\draw (1, 0) --(1, 9);
\draw (13, 0) --(13, 9);
\draw (5,9) arc (180:360:2);
\end{tikzpicture}
\caption{Diagrams for $\eps_2^4$ and $\delta_2^4$, respectively, in $\hom([4], [2])$. }
\label{fig:gen}
\end{figure}

Every diagram in the Temperley-Lieb category can be decomposed into certain generators, as described by the following theorem. 

\begin{thm}[\cite{cautis}]\label{tlcat}
The algebra $\bigoplus_{m, n \in \ZZ_{\geq 0}} \hom([m], [n])$ is generated as an algebra by all $\eps_i^n$ and $\delta_i^n$, where $1 \leq i \leq n-1$, with the following relations: 
\begin{enumerate}
\item[(i)] $\eps_i^n \cdot \delta_i^n = -q-q^{-1}$, and
\item[(ii)] $(\id \otimes\, \eps_i^n) \cdot (\delta_i^n \otimes \id) = \id = (\eps_i^n \otimes \id) \cdot (\id \otimes\, \delta_i^n).$ 
\end{enumerate}
\end{thm}

The first of these relations is the removal of contractible loops. Looking back at the Temperley-Lieb algebra, it is apparent that we set $\beta = -q-q^{-1}$ since that is the factor we multiply by when obtaining a contractible loop. The second relation (see Figure~\ref{fig:TLsecondrelation}) is the fact that a curved link can be contracted to a straight link. 

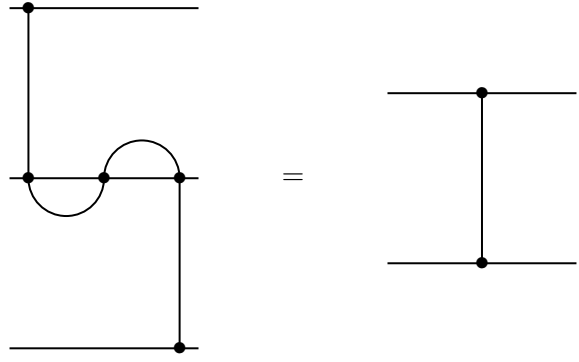
\begin{figure}[!ht]
\centering
\tikzset{every picture/.style={line width=0.75pt}} 

\begin{tikzpicture}[scale=.25]
\node at (1,18){$\bullet$};
\node at (1,9){$\bullet$};
    \node at (5,9){$\bullet$};
    \node at (9,9){$\bullet$};
    \node at (9,0){$\bullet$};
\draw (0,0) -- (10,0);
\draw (0,9) -- (10,9);
\draw (9, 0) -- (9, 9);
\draw (5,9) arc (360:180:2);
\draw (0,18)--(10,18);
\draw(1, 9)--(1, 18);
\draw (5, 9) arc(180:0:2);
\node at (15, 9) {=};
\draw (20,4.5) -- (30,4.5);
\draw (20,13.5) -- (30,13.5);
\draw(25, 4.5)--(25, 13.5);
\node at (25,4.5){$\bullet$};
\node at (25,13.5){$\bullet$};
\end{tikzpicture}
\caption{The second relation in Theorem~\ref{tlcat}. }
\label{fig:TLsecondrelation}
\end{figure}

With these relations, we can describe the spin representation diagrammatically. 

\subsection{The Spin Representation}
In this section, we describe the spin representation $(\CC^2)^{\otimes n}$ and how the Temperley-Lieb algebra acts on it. 

Let $v_+$ denote $(1, 0) \in \CC^2$ and $v_-$ denote $(0, 1) \in \CC^2$. The 4-dimensional vector space $\CC^2 \otimes \CC^2$ has basis $\{v_+ \otimes v_+, v_+ \otimes v_-, v_- \otimes v_+, v_- \otimes v_-\}$. 

Let $\epsilon\colon\CC^2\otimes\CC^2\to\CC$ be \[v_+ \otimes v_+ \mapsto 0, v_+ \otimes v_- \mapsto -q, v_- \otimes v_+ \mapsto 1, v_- \otimes v_- \mapsto 0\] and let $\delta: \CC \rightarrow \CC^2\otimes\CC^2$ be \[ 1 \mapsto v_+ \otimes v_- - q^{-1}v_- \otimes v_+.\] 

Let $\epsilon_i^n\colon(\CC^2)^{\otimes n}\to (\CC^2)^{\otimes n-2}$ be $(\id)^{\otimes i-1} \otimes \eps \otimes (\id)^{\otimes n-i-1}$ and $\delta_i^n\colon(\CC^2)^{\otimes n-2}\to (\CC^2)^{\otimes n}$ be $\id^{\otimes i-1} \otimes \,\delta\, \otimes \id^{\otimes n-i-1}$. 

Since $\epsilon_i^n$ and $\delta_i^n$ generate $\bigoplus_{m, n \geq 0}\hom([m],[n])$ and all the relations in Theorem~\ref{tlcat} are satisfied, we have an action of $\bigoplus_{m, n \geq 0}\hom([m],[n])$ on $\bigoplus_{n \geq 0}(\CC^2)^{\otimes n}$.

We can consider the case of $m = n$ to obtain a representation of $\TL_n$. 
\begin{defn}
The \emph{spin representation} is the action of $\TL_n = \hom([n], [n])$ on $(\CC^2)^{\otimes n}$. 
\end{defn}

The following example illustrates how we can view the action of the spin representation through slicing a diagram. 
\begin{ex}
We describe an example of how $\TL_4$ acts on $(\CC^2)^{\otimes 4}$. Given the first diagram in Figure~\ref{fig:slice}, we write it in terms of the generators of $\TL$ by decomposing it as shown in the second and third diagrams. Explicitly, the first diagram decomposes into $\delta_1^4\delta_3^2\eps_1^2\eps_2^4$. 

\begin{figure}[!ht]
\centering
\tikzset{every picture/.style={line width=0.75pt}} 
\begin{tikzpicture}[scale=.25]
    \node at (1,9){$\bullet$};
    \node at (5,9){$\bullet$};
    \node at (9,9){$\bullet$};
    \node at (13,9){$\bullet$};
    \node at (1,0){$\bullet$};
    \node at (5,0){$\bullet$};
    \node at (9,0){$\bullet$};
    \node at (13,0){$\bullet$};
\draw (0,0) -- (14,0);
\draw (0,9) -- (14,9);
\draw (5,0) arc (180:0:2);
\draw (1,9) arc (180:360:2);
\draw (9,9) arc (180:360:2);
\begin{scope}
\clip (0,0) rectangle (14,9);
\draw (7,0) ellipse (6 and 4);
\end{scope}
\end{tikzpicture}
\hspace{10mm}
\begin{tikzpicture}[scale=.25]
    \node at (1,9){$\bullet$};
    \node at (5,9){$\bullet$};
    \node at (9,9){$\bullet$};
    \node at (13,9){$\bullet$};
    \node at (1,0){$\bullet$};
    \node at (5,0){$\bullet$};
    \node at (9,0){$\bullet$};
    \node at (13,0){$\bullet$};
\draw (0,0) -- (14,0);
\draw (0,9) -- (14,9);
\draw (5,0) arc (180:0:2);
\draw (1,9) arc (180:360:2);
\draw (9,9) arc (180:360:2);
\begin{scope}
\clip (0,0) rectangle (14,9);
\draw (7,0) ellipse (6 and 4);
\end{scope}
\draw[red][dotted] (0, 2.5) -- (14, 2.5);
\draw[red][dotted] (0, 3.7) -- (14, 3.7);
\draw[red][dotted] (0, 6) -- (14, 6);
\end{tikzpicture}
\hspace{10mm}
    \begin{tikzpicture}[scale=.25]
    \node at (1,9){$\bullet$};
    \node at (5,9){$\bullet$};
    \node at (9,9){$\bullet$};
    \node at (13,9){$\bullet$};
    \node at (1,0){$\bullet$};
    \node at (5,0){$\bullet$};
    \node at (9,0){$\bullet$};
    \node at (13,0){$\bullet$};
\draw (0,0) -- (14,0);
\draw (0,9) -- (14,9);
\draw (5,0) arc (180:0:2);
\draw (1,9) arc (180:360:2);
\draw (9,9) arc (180:360:2);
    \node at (2.5,2.5){$\bullet$};
    \node at (11.5,2.5){$\bullet$};
\draw (1, 0) -- (2.5, 2.5);
\draw (13, 0) -- (11.5, 2.5);
\draw (2.5, 2.5) -- (5, 3.7);
\draw (11.5, 2.5) -- (9, 3.7);
    \node at (5,3.7){$\bullet$};
    \node at (9,3.7){$\bullet$};
\draw (0, 2.5) -- (14, 2.5);
\draw (0, 3.7) -- (14, 3.7);

\begin{scope}
        \clip (0,3.7) rectangle (14,9);
        \draw (7,3.7) ellipse (2 and 0.5);
        \end{scope}
\draw (0, 6) -- (14, 6);
\end{tikzpicture}
\caption{An example of how a $\TL_4$ diagram is sliced.}
\label{fig:slice}
\end{figure}
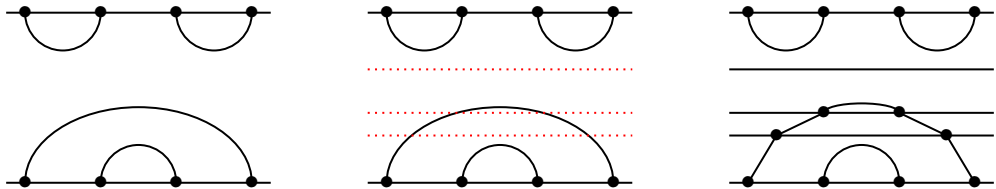

The action of the diagram in Figure~\ref{fig:slice} is as follows: 
\[\begin{tikzcd}
	{\mathbb C^2 \otimes \mathbb C^2 \otimes \mathbb C^2 \otimes \mathbb C^2} & {\mathbb C^2 \otimes \mathbb C^2} & {\mathbb C^2 \otimes \mathbb C^2} & {\mathbb C} & {\mathbb C^2 \otimes \mathbb C^2 \otimes \mathbb C^2 \otimes \mathbb C^2}
	\arrow["{\mathrm{id} \otimes \epsilon \otimes \id}", from=1-1, to=1-2]
	\arrow["{\mathrm{id} \otimes \id}", from=1-2, to=1-3]
	\arrow["\epsilon", from=1-3, to=1-4]
	\arrow["{\delta \otimes \delta}", from=1-4, to=1-5]
\end{tikzcd}. \]
\end{ex}
\begin{rem}
The reason the spin representation is defined in this way is that the maps $\eps$ and $\delta$ are in fact homomorphisms of $\quantum$-modules (see Section~\ref{cor1.10section}). 
\end{rem}

We can also define the spin representation in terms of its basis $v_- \otimes v_-$, $v_- \otimes v_+$, $v_+ \otimes v_-$, and $v_+ \otimes v_+$. 
\begin{lem}[\cite{degroot}]\label{spinreplem}
The $\CC$-module $(\CC^2)^{\otimes n}$ equipped with the map $\zeta : \TL_n \rightarrow \End((\CC^2)^{\otimes n} )$ where
\[ \zeta(e_i) = \displaystyle \begin{pmatrix}
0 & 0 & 0 & 0\\
0 & -q & 1 & 0\\
0 & 1 & -q^{-1} & 0\\
0 & 0 & 0 & 0
\end{pmatrix}_{i, i+1}\] is the spin representation. Here, $B_{i, i+1}$, where $B$ is a matrix, denotes $B$ acting on the $i$ and $(i+1)$th components. 
\end{lem}
\begin{rem}
Note that if we set $\delta = \displaystyle \begin{pmatrix}
0\\
1\\
-q^{-1}\\
0
\end{pmatrix}$ and $\eps = \begin{pmatrix}
0 & -q & 1 & 0
\end{pmatrix}$, we obtain that $\zeta(e_i) = \delta_i^n \circ \eps_i^n$, which is consistent with our definition of the spin representation. 
\end{rem}

\section{Spherical and Aspherical Modules}
In this section, we introduce spherical and aspherical modules, which are specific modules induced from the trivial and sign representations on a subalgebra of the Hecke algebra. The spherical and aspherical modules also have a canonical basis, similarly defined using the bar involution and the upper triangularity property. 

We define an action of $\mH$ on $(\CC^2)^{\otimes n}$. To define the action of $\mathcal H$, we must first define the \emph{$R$-matrix} $\check R_{11}: \CC^2 \otimes \CC^2 \rightarrow  \CC^2 \otimes \CC^2$ by $\check R_{11} = q^{1/2} \delta \circ \eps + q^{-1/2} \id$. Note that $\check R_{11}$ is a nontrivial isomorphism of $\CC^2 \otimes \CC^2$. Now, we can define $\mathcal H \rightarrow \End((\CC^2)^{\otimes n})$ by \[H_i \mapsto 1^{\otimes(i-1)}\otimes q^{-1/2}\check{R}_{11}\otimes 1^{\otimes(n-i-1)}.\]
\begin{rem}
When $q = 1$, this is just the natural action of $S_n$ on $(\CC^2)^{\otimes n}$ by permuting components. 
\end{rem}

Given a subset $J \subseteq S$, let the corresponding coxeter group be $W_J \subseteq W$. 

\begin{defn}
Let $W^J \subseteq W$ denote the set of elements of minimal length in each left $W_J$-coset. 
\end{defn}
We have a bijection $W_J \times W^J \cong W$ by $(x, y) \mapsto xy$. Denote $\mathcal H_J  = \mathcal H(W_J)$. Let $\CC_\mathrm{triv}$ denote the $\mH_J$-module mapping $H_s$ to $q^{-1}$ for every $s \in J$ and let $\CC_\mathrm{sgn}$ denote the $\mH_J$-module mapping $H_s$ to $-q$ for every $s \in J$ 

\begin{defn}[Spherical and Aspherical Modules]
The \emph{spherical module} is \[\mathcal M = \Ind_{\mathcal H_J}^{\mathcal H} \CC_\mathrm{triv} =  \mathcal H \otimes_{\mathcal H_J} \CC_\mathrm{triv}\] while the \emph{aspherical module} is  \[\mathcal N = \Ind_{\mathcal H_J}^{\mathcal H} \CC_\mathrm{sgn} =  \mathcal H \otimes_{\mathcal H_J} \CC_\mathrm{sgn}.\]
\end{defn}

The spherical module has a basis similar to the Kazhdan-Lusztig Basis of the Hecke algebra. The standard bases are $M_x = H_x \otimes 1 \in \mathcal M$ and $N_x = H_x \otimes 1 \in \mathcal N$ for $x \in W^J$. The bar involution on $\mathcal H$ is compatible with these bases, by $\overline{h \otimes 1} = \overline h \otimes 1$ for $h \in \mathcal H$. 

Similar to Lemma~\ref{heckeselfdual} and Theorem~\ref{klbasis}, we have the following. 

\begin{thm}[\cite{deodhar}]
For all $x \in W^J$, there exists a unique self-dual element $\underline{M}_x \in \mathcal M$ such that $\underline{M}_x \in M_x + \sum_y q^{-1}\ZZ[q^{-1}] M_y$. We can also do the same for $\mathcal N$, obtaining elements $\underline N_x$. These form a basis of $\mathcal M$ and $\mathcal N$, respectively. 

In fact, $\underline{M}_x \in M_x + \sum_{y\prec x} q^{-1}\ZZ[q^{-1}] M_y$, and the same also holds for $\mathcal N$. 
\end{thm}
\begin{rem}
Similarly to the canonical basis of the Hecke algebra, we can use $q\ZZ[q]$ instead of $q^{-1}\ZZ[q^{-1}]$. 
\end{rem}

\section{The Dual Canonical Basis Through the Temperley-Lieb Algebra}\label{bijectionsection}
In Section~\ref{heckesection}, we defined a canonical basis for the Hecke algebra that is invariant under an involution and satisfies upper triangularity properties. Furthermore, each canonical basis element corresponded to an element in the standard basis. We can construct an analog for $(\CC^2)^{\otimes n}$. A standard basis is $\{v_{\pm} \otimes v_{\pm} \otimes \dots \otimes v_{\pm}\}$. We define Lusztig's dual canonical basis as follows. 

\begin{thm}[Dual Canonical Basis \cite{khovanov}]
Given $v_{k_n} \otimes v_{k_{n-1}} \otimes \dots \otimes v_{k_1}$, where $k_i \in \{-, +\}$ for all $i$, there exists a unique $v_{k_n}\heart\dots\heart v_{k_1} \in (\CC^2)^{\otimes n}$ such that 
\begin{enumerate}
\item[(i)] $v_{k_n}\heart\dots\heart v_{k_1}$ is self dual under the bar involution (see \cite{khovanov}), 
\item[(ii)] $v_{k_n}\heart\dots\heart v_{k_1}-v_{k_n}\otimes\dots\otimes v_{k_1}$ is in $q^{-1}\ZZ[q^{-1}]$, and
\item[(iii)] If we consider $+$ to have value $+1$ and $-$ to have value $-1$, $v_{k_n}\heart\dots\heart v_{k_1}$ is equal to $v_{k_n}\otimes\dots\otimes v_{k_1}$ plus a linear combination of elements $v_{l_n}\otimes\dots\otimes v_{l_1}$ such that $\sum_{i = 1}^n l_i = \sum_{i = 1}^n k_i$ and $\sum_{i = 1}^{n'} l_i = \sum_{i = 1}^{n'} k_i$ for all $1 \leq n' \neq n$. 
\end{enumerate}
The set of elements $v_{k_n}\heart\dots\heart v_{k_1} \in (\CC^2)^{\otimes n}$ forms a basis of $(\CC^2)^{\otimes n}$ called the \emph{dual canonical basis}. 
\end{thm}
\begin{rem}
The first condition is the self-dual condition, while the second and third conditions give the upper triangularity properties. 
\end{rem}
\subsection{Relation of $(\CC^2)^{\otimes n}$ to Induced $\TL_n$-modules}
We introduce how the diagrammatic basis of the Temperley-Lieb algebra relates to the dual canonical basis in this subsection. For the rest of this paper, let $J = \{ s_1, \dots, s_{k-1}, s_{k+1}, \dots, s_{n-1}\}$, so $\mH_J$ is precisely $\mH_k \otimes \mH_{n-k}$. Furthermore, let $(\CC^2)^{\otimes n}_k$ denote the subspace of $(\CC^2)^{\otimes n}$ generated by the standard basis elements where $k$ of the components are $v_-$ and the rest are $v_+$. 

\begin{prop}[\cite{fkk}] \label{fkk}
There is an embedding of $\mH$-modules $\iota\colon \mM \hookrightarrow (\CC^2)^{\otimes n}$ that maps the canonical basis of the left hand side to the subset of the dual canonical basis of the right hand side that lies in $(\CC^2)^{\otimes n}_k$ by $\iota\colon H_{s_i} \otimes 1 \mapsto H_{s_i} \cdot \underbrace{v_- \otimes \dots \otimes v_-}_k \otimes \underbrace{v_+ \otimes \dots \otimes v_+}_{n-k}$. 
\end{prop}
\begin{lem}[\cite{degroot}]\label{factors}
The action of $\mathcal H$ on $(\CC^2)^{\otimes n}$ factors through the action of $\TL_n$ by $\phi_q\colon H_i \mapsto e_i + q^{-1}$. 
\end{lem}


We present one more general lemma of induced representations that will be useful later in the section. 

\begin{lem}\label{indlemma}
Let $A$ be a subalgebra of algebra $B$, let $\phi\colon B \rightarrow B'$ be a homomorphism of algebras, and define $A' = \phi(A)$. Let $V$ be a $A'$-module. There exists a map $\Ind_A^B V \rightarrow \Ind_{A'}^{B'} V$ that is surjective if $\phi$ is surjective. 
\end{lem}
\begin{proof}
By definition, $\Ind_A^B$ is $B \otimes_A V$ and $\Ind_{A'}^{B'} V$ is $B' \otimes_{A'} V$, which are quotients of $B \otimes V$ and $B' \otimes V$, respectively. A map $B \otimes V \rightarrow B' \otimes V$ that respects the relations is $b \otimes v \mapsto \phi(b) \otimes v$. 
\end{proof}

With these results, we can show the following. 
\begin{lem} \label{indlem}
We have an isomorphism of induced modules
\[ \mM = \heckeinduced \cong \TLinduced.\]
\end{lem}
\begin{proof}
Let $\iota$ denote the embedding described in Proposition~\ref{fkk}. Let $\iota^{-1}\colon \im(\iota) \rightarrow \mM$ be the inverse of $\iota$. 

First, we construct a surjective map \[\Phi\colon \heckeinduced \twoheadrightarrow \TLinduced, \] which can be done by applying Lemma~\ref{indlemma} with the surjection $\phi_q\colon \mathcal H_n \twoheadrightarrow \TL_n$. 

We construct another surjective map \[\Psi\colon \TLinduced \twoheadrightarrow \heckeinduced.\] Due to the universal property of induced representations and Lemma~\ref{factors}, such a map is determined by a map in $\hom_{\TL_k \otimes \TL_{n-k}}(\CC_\mathrm{triv}, \mM)$, where the action of $\TL_k \otimes \TL_{n-k}$ on $\mM$ is due to $\iota$. We can verify that one such map is \[\psi\colon 1 \mapsto \iota^{-1}(\underbrace{v_- \otimes \dots \otimes v_-}_k \otimes \underbrace{v_+ \otimes \dots \otimes v_+}_{n-k}).\] Then, \[\Psi\colon g \otimes 1 \mapsto g\cdot \psi(1)\] for $g \in \TL_n$. This map is a surjection as $(e_i +q^{-1}) \otimes 1 \mapsto (e_i +q^{-1}) \cdot \psi(1) = H_{s_i} \cdot \psi(1)$ for all $1 \leq i \leq n$, which generates $\mM$. 

Note that both $\Phi$ and $\Psi$ are homomorphisms of $\mathcal H_n$-modules. Furthermore, in both representations, the vectors $1 \otimes 1$ are both cyclic. Both $\Phi$ and $\Psi$ send identity to identity, so $\Phi \circ \Psi = \Psi \circ \Phi = \id$. 
\end{proof}

We need the following result to explicitly identify the bases of the induced modules given above. 

\begin{prop}[Fan-Green \cite{fg}] \label{fg}
The surjection $\phi_q\colon \mathcal H \rightarrow \TL_n$ by $H_{s_i} \mapsto e_i + q^{-1}$ maps a canonical basis element to either an element of the diagrammatic basis or $0$. 
\end{prop}

Now, we have the following. Let $W_J$ denote $S_k \times S_{n-k}$ and let $W^J$ be the set of minimal-length elements of each $W_J$-coset. 
\begin{lem} \label{cantocan}
For $w \in W^J$, the surjection $\phi\colon \mathcal H \twoheadrightarrow \Ind_{\mathcal H_k \otimes \mathcal H_{n-k}}^{\mathcal H_n} \CC_\mathrm{triv}$ by $H \mapsto H \otimes 1$ sends $B_w$ to the canonical basis element $\underline M_w$. 
\end{lem}
\begin{proof}
The canonical basis element $B_w$ is in $H_w + \sum_{y\prec w} q^{-1} \ZZ[q^{-1}]H_y$. The element $H_w$ maps to $H_w \otimes 1 = M_w$. For each $H_y$ where $y \prec w$, $\phi$ sends it to an element of the form $(-q^{-1})^a M_x$, where $a$ is a nonnegative integer and $x$ is the element in the $W_J$-coset of $y$ with minimal length, so $x \preceq y$. Therefore, $\phi(B_w) \in M_w + \sum_{y \prec w}q^{-1}\ZZ[q^{-1}]M_y$, and since $\phi(B_w)$ is also self-dual, it is equal to $\underline M_x$ by uniqueness. 
\end{proof}
\begin{lem} \label{basisidentification}
The isomorphism given by Lemma~\ref{indlem} identifies the canonical basis of $\mM$ with the diagrammatic basis of $\TLinduced$. 
\end{lem}
\begin{proof}
The dual canonical basis of $\mM$ consists of elements of the form $B_w \otimes 1$, where $B_w$ is a canonical basis element of $\mathcal H$, by Lemma~\ref{cantocan} combined with the fact that the canonical basis elements of the spherical module are $\underline M_w$ for $w \in W^J$. The bijection $\Phi\colon \mM \twoheadrightarrow \TLinduced$ described in Lemma~\ref{indlem} maps $B_w \otimes 1$ to $\phi_q(B_w) \otimes 1$, where $\phi_q(B_w)$ is a diagram in $\TL_n$ by Proposition~\ref{fg}. Each $\phi_q(B_w) \otimes 1$ is zero if $\phi_q(B_w) \in \TL_k \otimes \TL_{n-k}$ and a diagrammatic basis element of $\TLinduced$ otherwise. However, due to the fact that $\Phi$ is a bijection, $\phi_q(B_w) \otimes 1$ cannot be zero, so it is an element of the diagrammatic basis. Therefore, we have the identification as described in the lemma. 
\end{proof}

\begin{thm} 
There is an embedding \[ \TLinduced \hookrightarrow (\CC^2)^{\otimes n}\] that identifies the diagrammatic basis to a subset of the dual canonical basis of $(\CC^2)^{\otimes n}$ lying in $(\CC^2)^{\otimes n}_k$. 
\end{thm}
\begin{proof}
Combining Lemma~\ref{indlem} and Lemma~\ref{basisidentification} with Proposition~\ref{fkk}, we obtain an embedding $b\colon \TLinduced \hookrightarrow (\CC^2)^{\otimes n}$, shown in the following commutative diagram. 

\[\begin{tikzcd}
	{\mM} & {\mathrm{Ind}_{\mathrm{TL}_k \otimes \mathrm{TL}_{n-k}}^{\mathrm{TL}_n} \mathbb C_\mathrm{triv}} \\
	& {(\mathbb C^2)^{\otimes n}}
	\arrow["\sim", from=1-1, to=1-2]
	\arrow["\iota"', from=1-1, to=2-2]
	\arrow["b", from=1-2, to=2-2]
\end{tikzcd}\]
As in Proposition~\ref{fkk}, $\iota$ is an injection sending canonical basis elements of $\mM$ to dual canonical basis elements in $(\mathbb C^2)^{\otimes n}$. Since the isomorphism described in Lemma~\ref{indlem} identifies dual canonical basis elements to diagrammatic basis elements by Lemma~\ref{basisidentification}, as $b$ is the unique map making the diagram commutative, $b$ must be an injection mapping diagrammatic basis elements to dual canonical basis elements, proving the claim. 
\end{proof}
\begin{lem} \label{tlactionlem}
The action of $\TL_n$ on $(\CC^2)^{\otimes n}$ either sends $v_- \otimes \dots \otimes v_- \otimes v_+ \otimes \dots \otimes v_+$, with $v_-$ appearing $k$ times and $v_+$ appearing $n-k$ times, to $0$ or to a vector in $(\CC^2)^{\otimes n}_k$. 
\end{lem}
\begin{proof}
We show that the action of $\id$, $\eps$, and $\delta$ (corresponding to the three possible features) all satisfy the lemma. The action of $\id$ keeps the number of $v_-$ and $v_+$ constant. The number of actions of $\eps$ and $\delta$ must be equal since there are an equal number of quasi-simple links on the top and the bottom. The action of $\eps$ sends $v_+ \otimes v_+$ and $v_- \otimes v_-$ to $0$, so in this case, the action of the entire diagram containing this feature is $0$. If $\eps$ acts on $v_- \otimes v_+$ or $v_+ \otimes v_-$, it removes one $v_-$ and one $v_+$, and $\delta$ adds one $v_-$ and one $v_+$ in each element of the sum, so after all the actions, the number of $v_-$ and $v_+$ will remain constant. 
\end{proof}

\begin{prop}\label{tlactionprop}
The following $\TL_n$-modules are isomorphic:
\[\bigoplus_{0 \leq k \leq n} \Ind_{\TL_k \otimes \TL_{n-k}}^{\TL_n}\CC_\mathrm{triv} \cong (\CC^2)^{\otimes n}\]
and the isomorphism identifies the diagrammatic basis of the left hand side to the dual canonical basis of the right hand side. 
\end{prop}
\begin{proof}
Due to Lemma~\ref{tlactionlem} and by definition of the dual canonical basis, the subset of the dual canonical basis described in Proposition~\ref{tlactionprop} is precisely the elements $v_{k_n}\heart\dots\heart v_{k_1}$ that are in $(\CC^2)^{\otimes n}_k$. The collection of these for all $0 \leq k \leq n$ is precisely the dual canonical basis. 
\end{proof}
\subsection{Explicit Bijection Between Diagrammatic Basis and Dual Canonical Basis}\label{diagrambij}
Given a label of the dual canonical basis $v_{k_n} \heart \dots \heart v_{k_1}$ in $(\CC^2)^{\otimes n}_k$, we construct an unique diagram $D$ such that the bijection described in Proposition~\ref{tlactionprop} identifies $D$ with $v_{k_n} \heart \dots \heart v_{k_1}$. 

First, we draw the quasi-simple links on the top line. Label the vertices from left to right by $v_{k_n}$, $v_{k_{n-1}}$, $\dots$, $v_{k_1}$. We construct these quasi-simple links by the following procedure, an example of which is shown in Figure~\ref{fig:linkstatecon}:
\begin{enumerate}
\item Consider the least $i$ such that $k_i = +$ that we have not considered yet. 
\begin{enumerate}
\item[$\bullet$] If there is no $j < i$ such that $k_j = -$, move on to step (2). 
\item[$\bullet$] If $j$ is the greatest $j < i$ such that $k_j = -$, draw a quasi-simple link between $v_{k_i}$ and $v_{k_j}$ and move on to step (2). 
\end{enumerate}
\item Repeat step (1) until we have considered all $v_+$. 
\end{enumerate}

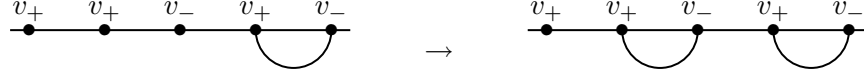
\begin{figure}[!ht]
\centering
\tikzset{every picture/.style={line width=0.75pt}} 

\begin{tikzpicture}[scale=.25]
\draw (0,9) -- (18,9);
\draw (13,9) arc (180:360:2);
\node at (1,9){$\bullet$};
    \node at (5,9){$\bullet$};
    \node at (9,9){$\bullet$};
    \node at (13,9){$\bullet$};
        \node at (17,9){$\bullet$};
\node at (1, 10) {$v_+$};
\node at (5, 10) {$v_+$};
\node at (9, 10) {$v_-$};
\node at (13, 10) {$v_+$};
\node at (17, 10) {$v_-$};
\end{tikzpicture}
\hspace{5mm}
\begin{tikzpicture}[scale=.25]
\node at (0, 0) {$\rightarrow$};
\end{tikzpicture}
\hspace{5mm}
\begin{tikzpicture}[scale=.25]
\node at (1,9){$\bullet$};
    \node at (5,9){$\bullet$};
    \node at (9,9){$\bullet$};
    \node at (13,9){$\bullet$};
        \node at (17,9){$\bullet$};
\draw (0,9) -- (18,9);
\draw (5,9) arc (180:360:2);
\draw (13,9) arc (180:360:2);
\node at (1, 10) {$v_+$};
\node at (5, 10) {$v_+$};
\node at (9, 10) {$v_-$};
\node at (13, 10) {$v_+$};
\node at (17, 10) {$v_-$};
\end{tikzpicture}
\caption{Construction of the link state of $v_+ \heart v_+ \heart v_- \heart v_+ \heart v_-$. }
\label{fig:linkstatecon}
\end{figure}

We claim that this procedure gives a valid cup diagram of the top line. First of all, we claim that there are no unlinked $v_{k_m}$ such that $v_{k_i}$ is linked to $v_{k_j}$ and $j < m < i$. Since $j < i$, we have $k_i = +$ and $k_j = -$. If $k_m = +$, $v_{k_m}$ would have been linked with $v_{k_j}$; if $k_m = -$, then $v_{k_i}$ would have been linked with $v_{k_m}$. Furthermore, there are no intersecting links as all of the quasi-simple links we constructed connect the nearest unlinked vertices. By the bijection of cup diagrams and parenthesis diagrams given in Lemma~\ref{cupparenthesis}, this is sufficient. 

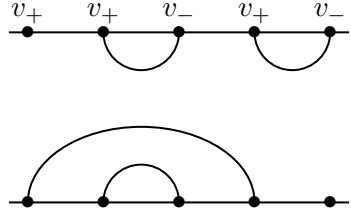
\begin{figure}[!ht]
\centering
\tikzset{every picture/.style={line width=0.75pt}} 
\begin{tikzpicture}[scale=.25]
\node at (1,9){$\bullet$};
    \node at (5,9){$\bullet$};
    \node at (9,9){$\bullet$};
    \node at (13,9){$\bullet$};
        \node at (17,9){$\bullet$};
\node at (1,0){$\bullet$};
    \node at (5,0){$\bullet$};
    \node at (9,0){$\bullet$};
    \node at (13,0){$\bullet$};
        \node at (17,0){$\bullet$};
\draw (0,0) -- (18,0);
\draw (0,9) -- (18,9);
\draw (5,0) arc (180:0:2);
\draw (5,9) arc (180:360:2);
\draw (13,9) arc (180:360:2);
\begin{scope}
\clip (0,0) rectangle (18,9);
\draw (7,0) ellipse (6 and 4);
\end{scope}
\node at (1, 10) {$v_+$};
\node at (5, 10) {$v_+$};
\node at (9, 10) {$v_-$};
\node at (13, 10) {$v_+$};
\node at (17, 10) {$v_-$};
\end{tikzpicture}
\caption{Construction of the next step for the diagram corresponding to $v_+ \heart v_+ \heart v_- \heart v_+ \heart v_-$. }
\label{fig:nextstep}
\end{figure}

After this step, let the number of quasi-simple links on the top line be $p$. On the bottom line, we draw the quasi-simple link connecting the $(k-i)$th vertex to the $(k+1+i)$th for all $0 \leq i \leq p-1$, and so on, drawing $p$ quasi-simple links total. An example of this is shown in Figure~\ref{fig:nextstep}. 

Finally, we have $n-2p$ unlinked vertices on both the top and the bottom lines. We link the $i$th unlinked vertex from the left on the top line to the $i$th unlinked vertex from the left on the bottom line for $1 \leq i \leq n-2p$, obtaining our diagram. 

\begin{rem} \label{reverserem}
The above starts from a label of the dual canonical basis and obtains a diagram. We can also construct the inverse. Given a diagram $D$, we can obtain the label of the dual canonical basis we get from the action of it acting on $v_- \otimes \dots \otimes v_- \otimes v_+ \otimes \dots \otimes v_+$ just by observation. We do this as follows: 
\begin{enumerate}
\item Label the points on the bottom line of $D$ left to right from $a_1$ to $a_n$. The labels $a_1$ to $a_k$ correspond to $v_-$ while $a_{k+1}$ to $a_n$ correspond to $v_+$. At the end of these steps, we will label the top line with some permutation $a_{i_1}$, $\dots$, $a_{i_n}$ of $a_1$ to $a_n$, and obtain the label of the dual canonical basis from that. 
\item For the links connecting the $i$th vertex on the bottom line to the $j$ vertex at the top line, move $a_i$ to the $j$th position on the top line. 
\item Define an ordering left to right for quasi-simple links as based on the position of the left vertex of the link. Let there be $\ell$ quasi-simple links. For each $1 \leq i \leq \ell$, if the $i$th quasi-simple link on the bottom has labels $a_m$ and $a_n$, in that order, we label the left and right vertices of the $i$th quasi-simple link at the top with $a_n$ and $a_m$ respectively. 
\item Now that we have a labeling $a_{i_1}$, $a_{i_2}$, $\dots$, $a_{i_n}$ of the top line from left to right, the label of the dual canonical basis for the action of the diagram is $a_{i_1} \heart a_{i_2} \heart \dots \heart a_{i_n}$. 
\end{enumerate}
\end{rem}

\begin{ex}
The diagram shown in Figure~\ref{fig:labelcomp} acting on $v_- \otimes v_- \otimes v_+ \otimes v_+ \otimes v_+$ gives $v_+ \heart v_+ \heart v_- \heart v_+ \heart v_-$, since $a_1 = a_2 = v_-$ and $a_3 = a_4 = a_5 = v_+$ in the labelling of the top line of the figure. 

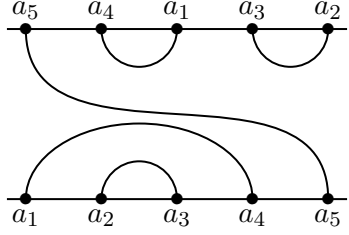
\begin{figure}[!ht]
\centering
\tikzset{every picture/.style={line width=0.75pt}} 

\begin{tikzpicture}[scale=.25]
\node at (1,9){$\bullet$};
    \node at (5,9){$\bullet$};
    \node at (9,9){$\bullet$};
    \node at (13,9){$\bullet$};
        \node at (17,9){$\bullet$};
\node at (1,0){$\bullet$};
    \node at (5,0){$\bullet$};
    \node at (9,0){$\bullet$};
    \node at (13,0){$\bullet$};
        \node at (17,0){$\bullet$};
\draw (0,0) -- (18,0);
\draw (0,9) -- (18,9);
\draw (5,0) arc (180:0:2);
\draw (5,9) arc (180:360:2);
\draw (13,9) arc (180:360:2);
\begin{scope}
\clip (0,0) rectangle (18,9);
\draw (7,0) ellipse (6 and 4);
\end{scope}
\draw (17,0) .. controls (17,8) and (1,1) .. (1,9);
\node at (1, -1) {$a_1$};
\node at (5, -1) {$a_2$};
\node at (9, -1) {$a_3$};
\node at (13, -1) {$a_4$};
\node at (17, -1) {$a_5$};
\node at (1, 10) {$a_5$};
\node at (5, 10) {$a_4$};
\node at (9, 10) {$a_1$};
\node at (13, 10) {$a_3$};
\node at (17, 10) {$a_2$};
\end{tikzpicture}
\caption{Computing the label for an element of the dual canonical basis given a diagram in $\Ind_{\TL_2 \otimes \TL_3}^{\TL_5} \CC_\mathrm{triv}$. }
\label{fig:labelcomp}
\end{figure}
\end{ex}
\begin{lem}
The construction above gives the bijection described in Proposition~\ref{tlactionprop}. 
\end{lem}
\begin{proof}
Given a label of the dual canonical basis $v_{k_n} \heart \dots \heart v_{k_1}$, let $D$ be the diagram constructed as above. By construction, the coefficient of $v_{k_n} \otimes \dots \otimes v_{k_1}$ of the action of $D$ on $\underbrace{v_- \otimes \dots \otimes v_-}_k \otimes \underbrace{v_+ \otimes \dots \otimes v_+}_{n-k}$ is equal to $1$. Since diagrammatic basis elements map to dual canonical basis elements and $v_{k_n} \heart \dots \heart v_{k_1}$ is the only dual canonical basis element where $v_{k_n} \otimes \dots \otimes v_{k_1}$ has a coefficient of $1$, we know that $D$ maps to $v_{k_n} \heart \dots \heart v_{k_1}$. 
\end{proof}

\begin{ex} \label{firstex}
The basis of $\Ind_{\TL_2 \otimes \TL_2}^{\TL_4}$ are the diagrams (three of which are shown in Figure~\ref{fig:indbasis}) that do not have simple arcs on the bottom line connecting the first and the second vertex, or the third and the fourth. 

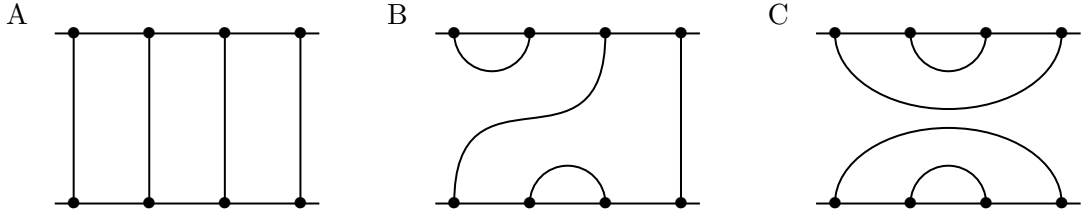
\begin{figure}[!ht]
\centering
\tikzset{every picture/.style={line width=0.75pt}} 
\begin{tikzpicture}[scale=.25]
    \node at (1,9){$\bullet$};
    \node at (5,9){$\bullet$};
    \node at (9,9){$\bullet$};
    \node at (13,9){$\bullet$};
    \node at (1,0){$\bullet$};
    \node at (5,0){$\bullet$};
    \node at (9,0){$\bullet$};
    \node at (13,0){$\bullet$};
\node at (-2, 10) {A};
\draw (0,0) -- (14,0);
\draw (0,9) -- (14,9);
\draw (1, 0)--(1, 9);
\draw (5, 0)--(5, 9);
\draw (9, 0)--(9, 9);
\draw (13, 0)--(13, 9);
\end{tikzpicture}
\hspace{5mm}
\begin{tikzpicture}[scale=.25]
    \node at (1,9){$\bullet$};
    \node at (5,9){$\bullet$};
    \node at (9,9){$\bullet$};
    \node at (13,9){$\bullet$};
    \node at (1,0){$\bullet$};
    \node at (5,0){$\bullet$};
    \node at (9,0){$\bullet$};
    \node at (13,0){$\bullet$};
\node at (-2, 10) {B};
\draw (0,0) -- (14,0);
\draw (0,9) -- (14,9);
\draw (1,9) arc (180:360:2);
\draw (1,0) .. controls (1,8) and (9,1) .. (9,9);
\draw (5,0) arc (180:0:2);
\draw (13, 0)--(13, 9);
\end{tikzpicture}
\hspace{5mm}
\begin{tikzpicture}[scale=.25]
    \node at (1,9){$\bullet$};
    \node at (5,9){$\bullet$};
    \node at (9,9){$\bullet$};
    \node at (13,9){$\bullet$};
    \node at (1,0){$\bullet$};
    \node at (5,0){$\bullet$};
    \node at (9,0){$\bullet$};
    \node at (13,0){$\bullet$};
\node at (-2, 10) {C};
\draw (0,0) -- (14,0);
\draw (0,9) -- (14,9);
\draw (5,0) arc (180:0:2);
\draw (5,9) arc (180:360:2);
\begin{scope}
\clip (0,0) rectangle (14,9);
\draw (7,0) ellipse (6 and 4);
\draw (7,9) ellipse (6 and 4);
\end{scope}
\end{tikzpicture}
\caption{Some diagrams in the basis of diagrams in $\Ind_{\TL_2 \otimes \TL_2}^{\TL_4} \CC_\mathrm{triv}$.}
\label{fig:indbasis}
\end{figure}

Explicitly computing the actions gives the following:
\begin{align*}
&A(v_- \otimes v_- \otimes v_+ \otimes v_+) = v_- \heart v_- \heart v_+ \heart v_+\\
&B(v_- \otimes v_- \otimes v_+ \otimes v_+) = v_+ \heart v_- \heart v_- \heart v_+ \\
&C(v_- \otimes v_- \otimes v_+ \otimes v_+) = v_+ \heart v_+ \heart v_- \heart v_-.
\end{align*}
Note that we do not have enough information to confirm whether the action of diagrams explicitly corresponds with the element of the dual canonical basis for this example yet; this example is purely to demonstrate the permutation of components of $v_- \heart v_- \heart v_+ \heart v_+$. The explicit computations will be in Example~\ref{exampletwo}. 
\end{ex}

\section{Computation of the Dual Canonical Basis}\label{computationdual}
\subsection{Inductive Computation of the Dual Canonical Basis}\label{computationinductive}
We use the results in the preceding section in order to explicitly compute the dual canonical basis of $(\CC^2)^{\otimes n}$. The following inductive computation (see Theorem \ref{thm1.8}) of the dual canonical basis is due to Khovanov, who has an involved proof in \cite{khovanov}. We first reprove this theorem using the results in the previous section. 

\begin{thm}[\cite{khovanov}] \label{thm1.8}
The dual canonical basis consists of elements of the form $v_{k_n} \heart \dots \heart v_{k_1}$ corresponding to $v_{k_n} \otimes \dots \otimes v_{k_1}$. They are constructed by the following rules: 
\begin{enumerate}
\item[(i)] $v_- \heart v_{k_{n-1}} \heart \dots = v_- \otimes (v_{k_{n-1}} \heart v_{k_{n-2}} \heart \dots )$, 
\item[(ii)] $v_{k_{n}} \heart \dots \heart v_{k_2} \heart v_+ = (v_{k_{n}} \heart \dots \heart v_{k_2}) \otimes v_+$, and
\item[(iii)] $\dots \heart v_{k_{i+1}} \heart v_+ \heart v_- \heart v_{k_{i-2}} \heart \dots = (\id^{\otimes(k-1)} \otimes \delta \otimes \id^{\otimes(n-k-1)}) (\dots \heart v_{k_{i+1}} \heart v_{k_{i-2}} \heart \dots)$. 
\end{enumerate}
\end{thm}

\begin{ex}\label{exampletwo}
We explicitly compute three elements of the dual canonical basis of $(\CC^2)^{\otimes 4}$, continuing our work from Example~\ref{firstex}. 

Assuming Theorem~\ref{thm1.8}, we obtain the following: 
\begin{enumerate}
\item[$\bullet$] $v_- \heart v_- \heart v_+ \heart v_+ = v_- \otimes v_- \otimes v_+ \otimes v_+$, 
\item[$\bullet$] $v_+ \heart v_- \heart v_- \heart v_+ = (v_+ \heart v_- \heart v_-) \otimes v_+ = v_+ \otimes v_- \otimes v_- \otimes v_+ - q^{-1}(v_- \otimes v_+ \otimes v_- \otimes v_+)$, and
\item[$\bullet$] $v_+ \heart v_+ \heart v_- \heart v_- = (\id \otimes \delta \otimes \id)(v_+ \heart v_-) = (\id \otimes \delta \otimes \id)(\delta) = (\id \otimes \delta \otimes \id)(v_+\otimes v_- - q^{-1}(v_- \otimes v_+)) = v_+ \otimes v_+ \otimes v_- \otimes v_- - q^{-1}(v_+ \otimes v_- \otimes v_+ \otimes v_-) - q^{-1}(v_- \otimes v_+ \otimes v_- \otimes v_+ - q^{-1}(v_- \otimes v_- \otimes v_+ \otimes v_+))$. 

\end{enumerate}
We confirm that the dual canonical basis aligns with the image of the action of $\Ind_{\TL_2 \otimes \TL_2}^{\TL_4} \CC_\mathrm{triv}$. Refer to Example~\ref{firstex} for the labels of the diagrams. We have the following actions: 

\begin{enumerate}
\item[$\bullet$] the action of $A$ is $v_- \otimes v_- \otimes v_+ \otimes v_+ \mapsto v_- \otimes v_- \otimes v_+ \otimes v_+$ which is equal to $v_- \heart v_- \heart v_+ \heart v_+$, 

\item[$\bullet$] the action of $B$ is $v_- \otimes v_- \otimes v_+ \otimes v_+ \mapsto v_- \otimes v_+ \mapsto v_+ \otimes v_- \otimes v_- \otimes v_+ - q^{-1}v_- \otimes v_+ \otimes v_- \otimes v_+$ which is equal to $v_+ \heart v_- \heart v_- \heart v_+$, and

\item[$\bullet$] the action of $C$ is $v_- \otimes v_- \otimes v_+ \otimes v_+ \mapsto v_- \otimes v_+ \mapsto 1 \mapsto v_+ \otimes v_- - q^{-1}v_- \otimes v_+ \mapsto v_+ \otimes v_+ \otimes v_- \otimes v_- - q^{-1}v_+ \otimes v_- \otimes v_+ \otimes v_- - q^{-1}(v_- \otimes v_+ \otimes v_- \otimes v_+ - q^{-1} v_- \otimes v_- \otimes v_+ \otimes v_+)$ which is equal to $v_+ \heart v_+ \heart v_- \heart v_-$. 
\end{enumerate}
Thus, we confirm that our claim in Example~\ref{firstex} is accurate. 
\end{ex}

\begin{thm}\label{diagramdualcomp}
The action of the diagram we obtain from the label of the dual canonical basis on $\underbrace{v_- \otimes \dots \otimes v_-}_k \otimes \underbrace{v_+ \otimes \dots \otimes v_+}_{n-k}$ is precisely the decomposition given in Theorem~\ref{thm1.8}. 
\end{thm}
\begin{proof}
All the quasi-simple links on the bottom line of the diagram can be ignored because $\eps\colon v_- \otimes v_+ \mapsto 1$. Label the vertices on the top line by $b_n$, $b_{n-1}$, $\dots$, $b_1$ from left to right. Let a quasi-simple link on the top line connect $b_j$ and $b_i$, where $j < i$. Then, there is a quasi-simple link on the top line connecting the $b_{j+1}$ and $b_{i-1}$, $b_{j+2}$ and $b_{i-2}$, and so on, until $b_{j+c}$ and $b_{i - c}$ where $i - c = j+c+1$. This phenomenon is shown in Figure~\ref{fig:linkstack}. 

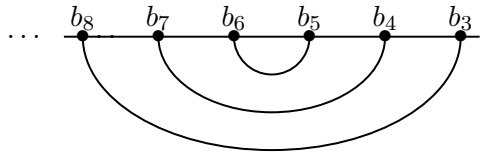
\begin{figure}[!ht]
\centering
\tikzset{every picture/.style={line width=0.75pt}} 

\begin{tikzpicture}[scale=.25]
    \node at (1,9){$\bullet$};
    \node at (5,9){$\bullet$};
    \node at (9,9){$\bullet$};
    \node at (13,9){$\bullet$};
        \node at (17,9){$\bullet$};
            \node at (21,9){$\bullet$};
\draw (0,9) -- (22,9);
\draw (9,9) arc (180:360:2);
\begin{scope}
\clip (22,9) rectangle (0,0);
\draw (11,9) ellipse (6 and 4);
\end{scope}
\begin{scope}
\clip (22,9) rectangle (0,0);
\draw (11,9) ellipse (10 and 6);
\end{scope}
\node at (1, 10) {$b_8$};
\node at (5, 10) {$b_7$};
\node at (9, 10) {$b_6$};
\node at (13, 10) {$b_5$};
\node at (17, 10) {$b_4$};
\node at (21, 10) {$b_3$};
\node at (-2, 9) {$\dots$};
\node at (2, 9) {$\dots$};
\end{tikzpicture}
\caption{Example of the nested quasi-simple links described above. }
\label{fig:linkstack}
\end{figure}

We have $v_{k_j} = v_{k_{j+1}} = \dots = v_{k_{j+c}} = v_-$ and $v_{k_i} = v_{k_{i-1}} = \dots = v_{k_{i - c}} = v_+$ from construction. In the action of the Temperley-Lieb algebra, $\delta$ will act on $b_{j+1}$ and $b_{i-1}$, $b_{j+2}$ and $b_{i-2}$, and so on, until $b_{j+c}$ and $b_{i - c}$. In the decomposition of Theorem~\ref{thm1.8}, in condition (iii), $\delta$ first appears for components $v_{k_{j+c}}$ and $v_{k_{i - c}}$, then $v_{k_{j+c-1}}$ and $v_{k_{i - c+1}}$, and so on, until $v_{k_{j+1}}$ and $v_{k_{i-1}}$. The quasi-simple links of a link state is a horizontal composition of these ``nested quasi-simple links.'' This shows that the $\delta$ appear in the same places for the action of the diagram and in the decomposition. 

 Now, for the diagram action, we only need to considering the links connecting the top line and the bottom line. Let there be $\alpha$ such links within the first $k$ vertices and $\beta$ within the next $n-k$ vertices. The action of just those non quasi-simple links map $\underbrace{v_- \otimes \dots \otimes v_-}_\alpha \otimes \underbrace{v_+ \otimes \dots \otimes v_+}_\beta$ to $\underbrace{v_- \otimes \dots \otimes v_-}_\alpha \otimes \underbrace{v_+ \otimes \dots \otimes v_+}_\beta$, but to a possibly different set of vertices from the top and bottom. 

Similarly, in the dual canonical basis, after condition (iii) is applied as many times as possible, the remaining label that is acted upon is $\underbrace{v_- \heart \dots \heart v_-}_\alpha \heart \underbrace{v_+ \heart \dots \heart v_+}_\beta$ since otherwise, (iii) can be further applied. This is just $\underbrace{v_- \otimes \dots \otimes v_-}_\alpha \otimes \underbrace{v_+ \otimes \dots \otimes v_+}_\beta$ by conditions (i) and (ii), which is consistent with the action of the diagram it is in bijection with. 
\end{proof}

Recall that Proposition~\ref{tlactionprop} gives a bijection between diagrams in $\Ind_{\TL_k \otimes \TL_{n-k}}^{\TL_n} \CC_\mathrm{triv}$ and the dual canonical basis elements $v_{k_n} \heart \dots \heart v_{k_1}$ with $k$ of the $k_i$ equal to $-$ and $n-k$ of the $k_i$ equal to $+$. We explicitly constructed this bijection and computed the action, proving that Theorem~\ref{thm1.8} holds through a Temperley-Lieb perspective. 

\begin{rem}
The last paragraph of page 6 of \cite{khovanov} states that the author doesn't know Proposition~\ref{fg} relates to his result of Theorem~\ref{thm1.8}. We have just resolved that. 
\end{rem}

\subsection{Explicit Formula for the Dual Canonical Basis}\label{computationexplicit}
Furthermore, using the results from Section~\ref{diagrambij}, we can write explicit formulas for the elements of the dual canonical basis. 
We generalize a method used in \cite{DNY} to study a different but related object in order to compute the dual canonical basis of $(\CC^2)^{\otimes n}$. Given a diagram $D$, label the vertices of both lines in $D$ from right to left. 

First, we state some definitions necessary in the explicit formula of the dual canonical basis. 
\begin{defn}
Given a diagram $D$ in $\Ind_{\TL_k\otimes \TL_{n-k}}^{\TL_n} \CC_\mathrm{triv}$ define $C(D)$ as
\[ C(D) = \{ (i, j) \mid 1 \leq i, j \leq n; i \text{ and } j \text{ are connected on the top line of diagram $D$}\}.\]
\end{defn}
\begin{defn}
If $D$ is a diagram with $m$ arcs on the top line, let $T(D)$ consist of all $m$-element subsets $S$ of $\{1, \dots, n\}$ such that $|S \cap C| = 1$ for each $C \in C(D)$. Let $\min(C)$ be the minimal element of $C$, and for each $S \in T(D)$, 
\[ \sgn(S) = (-q^{-1})^{c(S)} \text{ where } c(S) = |\{ C : C \in C(D), |C| = 2, S \cap C = \min(C)\}. \]
\end{defn}
\begin{defn}
For each $I \in T(D)$, define $v_I$, a vector in $(\CC^2)^{\otimes n}$ with the following procedure: 
\begin{enumerate}
\item For each link connecting the $i$th vertex on the bottom line to the $j$th vertex of the top line, set the $j$th component of $v_I$ equal to $v_-$ if $i \leq k$ and $v_+$ if $i > k$. 
\item For the components of $v_I$ that we have not set yet, set the $i$th position to $v_+$ if $i \in I$, and $v_-$ otherwise. 
\end{enumerate}
\end{defn}

Finally, we can describe the explicit formula for the dual canonical basis. 
\begin{thm}\label{dualcomputations}
For $D \in \Ind_{\TL_k\otimes \TL_{n-k}}^{\TL_n} \CC_\mathrm{triv}$, the dual canonical basis corresponding to $D$ is equal to 
\[ \sum_{I \in T(D)} \sgn(I)v_I. \]
\end{thm}
\begin{proof}
Consider $D$ embedded in category $\TL$. We will prove this lemma by induction on the number of arcs in the top line. Let there be $m$ arcs in the top line. In the base case of $m = 0$, the dual canonical basis element is $v_I = v_- \otimes \dots \otimes v_- \otimes v_+ \otimes \dots \otimes v_+$, which is consistent with the lemma as $T(D)$ is the empty set. For $D \in \Hom(n-2m, n)$, we can express $D = \delta_i D'$, where $D' \in \Hom(n-2m, n-2)$ and $\delta_i$ is $\delta$ acting after the $i$th component, representing an arc inserted between the $i$th and $i+1$th positions. We have 
\begin{align*}
D &= \delta_i D'\\
&= \sum_{I \in T(D')} \sgn(I)\delta_i v_I\\
&= \sum_{I \in T(D)}\sgn(I)v_I. \qedhere
\end{align*}
\end{proof}

\begin{rem}
Note that this explicit formula for the dual canonical basis was only possible due to the results in Section~\ref{diagrambij}, further demonstrating the significance of our interpretation via the Temperley-Lieb algebra. 
\end{rem}

\subsection{Explicit Formula for the Canonical Basis of the Spherical Module}
In \cite{khovanov}, Khovanov describes an isomorphism $\widetilde{Seq}\colon \mM \rightarrow (\CC^2)^{\otimes n}_k$ that identifies the canonical basis of the spherical module with the dual canonical basis of the spin representation (see Proposition~\ref{sphericalcanonical}). We can also explicitly compute the canonical basis of the spherical module as a corollary of Theorem~\ref{dualcomputations}. 

We construct an isomorphism $\widetilde{Seq}^{-1}\colon (\CC^2)^{\otimes n}_k \rightarrow \mathcal M$, which is the inverse of the map described in \cite{khovanov}. We use this isomorphism to explicitly calculate this canonical basis. 

\begin{defn}
Given a vector $v = v_{i_1} \otimes v_{i_2} \otimes \dots \otimes v_{i_n}$, where $k$ of the $i_j$ are $-$ and $n-k$ are $+$, let $\{a_1, \dots, a_k\}$ be the sequence where $i_{a_1} = \dots = i_{a_k} = -$ from left to right. Then, we define $\widetilde{Seq}^{-1}\colon v \mapsto M_w$ where $w = (s_{a_1-1}s_{a_1-2}\dots s_1)\dots  (s_{a_{k-1}-1}s_{a_{k-1}-2}\dots s_{k-1})(s_{a_k-1}s_{a_k-2}\dots s_k).$ 
\end{defn}
\begin{rem}
Note that this definition of  $\widetilde{Seq}^{-1}$ is indeed inverse to $\widetilde{Seq}$ as defined in \cite{khovanov} due to the explicit nature of the action. In our definition, each sequence $s_{a_i - 1}s_{a_i-2}\dots s_i$ moves one $v_-$ from the $i$th position into its correct position. 
\end{rem}
\begin{cor}\label{sphericalcor}
Let $I_0$ be the unique element in $T(D)$ where $\sgn(I_0) = 1$. Then, the canonical basis element of $\mathcal M$ with the first term $\widetilde{Seq}^{-1}(v_{I_0})$ is equal to $\sum_{I \in T(D)} \sgn(I)\widetilde{Seq}^{-1}(v_I)$. 
\end{cor}
\begin{proof}
The map $\widetilde{Seq}^{-1}$ is an explicit description of the inverse of $\widetilde{Seq}$ described in \cite{khovanov}, which identifies the Kazhdan-Lusztig basis of $\mathcal M$ with the dual canonical basis of $(\CC^2)^{\otimes n}$. As Theorem~\ref{dualcomputations} describes the dual canonical basis of $(\CC^2)^{\otimes n}$ and this corollary applies $\widetilde{Seq}^{-1}$ to it, the result is the canonical basis of $\mathcal M$. 
\end{proof}

\section{The Isomorphism $\End_\bU((\CC^2)^{\otimes n}) \cong ((\CC^2)^{\otimes 2n})^\bU$ Through $\TL_n$} \label{cor1.10section}
Let $\bU = \quantum$ denote the quantum group. In \cite{khovanov}, Khovanov gives the construction of the isomorphism $\End_\bU((\CC^2)^{\otimes n}) \cong ((\CC^2)^{\otimes 2n})^\bU$ using a complicated proof. We prove this construction using the bijection between the diagrammatic basis of the Temperley-Lieb algebra and the dual canonical basis that we explored in Section \ref{bijectionsection}. First, we recall a few facts about the quantum group $\bU = \quantum$. 

\begin{defn}[Quantum Group]
The quantum group $\quantum$ is an algebra over $\CC(q)$ with generators $E$, $F$, $K$, and $K^{-1}$ and relations 
\[KK^{-1}=1=K^{-1}K, \  KE = q^2EK, \ KF=q^{-2}FK\] and 
\[EF-FE=\frac{K-K^{-1}}{q-q^{-1}}. \]
\end{defn}

We can define the comultiplication $\bar \Delta$ in $\quantum$ by
\[\bar\Delta(K^{\pm 1}) = K^{\pm 1} \otimes K^{\pm 1}, \ \bar\Delta(E) = E \otimes 1 + K^{-1} \otimes E, \ \bar\Delta(F) = F \otimes K + 1 \otimes F\]
and the antipode $S$ by 
\[ S(K^{\pm 1}) = K^{\mp 1}, \ S(E) = -KE, \ S(F) = -FK^{-1}.\]

\begin{lem}[\cite{kassel}]
The following rules give an action of $\quantum$ on $\CC^2$: 
\[ K^{\pm 1}v_+= q^{\pm 1} v_+, \ K^{\pm 1}v_{-} = q^{\mp 1}v_{-}, \ Ev_+ = 0, \ Ev_{-} = v_+, \ Fv_+ = v_{-}, Fv_{-} = 0. \] 
\end{lem}

Let $A$ be a coalgebra. We assert a few general facts about $A$-modules, all of which follow in a standard way from the content of \cite{kassel}. 
\begin{lem}
If $V$ is a $A$ module, then $V^*$ is a $A$ module by $(a \cdot \alpha)(v) = \alpha(S(a) \cdot v)$ for $\alpha \in V^*$, $ a \in A$. 
\end{lem}
\begin{lem} \label{lem1}
There is an isomorphism of $A$-modules $W \otimes V^* \cong \Hom(V, W)$ given by $\phi_1\colon w \otimes \alpha \mapsto (v \mapsto \alpha(v)w)$. 
\end{lem}
%
\begin{lem} \label{lem2}
There is an isomorphism of $A$-modules $W^* \otimes V^* \cong (V \otimes W)^*$ given by $\phi_2\colon \alpha \otimes \beta \mapsto (v \otimes w \mapsto \alpha(w)\beta(v))$. 
\end{lem}
Finally, we prove one specific lemma for $\bU$. 
\begin{lem} \label{lem3}
There is an isomorphism of $\bU$-modules $\CC^2 \cong (\CC^2)^*$ given by $\phi_3\colon v^+ \mapsto v_-$ and $\phi_3\colon v^- \mapsto -q^{-1}v_+$. 
\end{lem}
\begin{proof}
We can show that the following defines an action on $(\CC^2)^*$: 
\[Kv^{\pm} = q^{\mp 1}v^\pm, K^{-1}v^{\pm} = q^{\pm}v^{\pm}\]
\[Ev^+ = -qv^-, Ev^- = 0\]
\[Fv^+ = 0, Fv^- = -q^{-1}v^+. \qedhere\]
\end{proof}
With these, we can construct an isomorphism $\End((\CC^2)^{\otimes n}) \cong (\CC^2)^{\otimes 2n}$ as follows. 
\[\begin{tikzcd}
	{\mathrm{End}((\CC^2)^{\otimes n})} & {(\CC^2)^{\otimes n} \otimes ((\CC^2)^{\otimes n})^*} & {(\CC^2)^{\otimes n} \otimes ((\CC^2)^*)^{\otimes n}} & {(\CC^2)^{\otimes 2n}}
	\arrow["{\varphi_1}", from=1-1, to=1-2]
	\arrow["{\varphi_2}", from=1-2, to=1-3]
	\arrow["{\varphi_3}", from=1-3, to=1-4]
\end{tikzcd}\]

Let $\phi_1$, denote the isomorphism defined in Lemma~\ref{lem1} with $V = W = (\CC^2)^{\otimes n}$, $A = \bU$. Let $\phi_2$ be the extension of the isomorphism in Lemma~\ref{lem2} to $n$ components, with each component being $\CC^2$ and $A = \bU$. Let $\phi_3$ be the isomorphism of the same name in Lemma~\ref{lem3}. 

\begin{prop} 
There is a isomorphism of $\bU$-modules $\End((\CC^2)^{\otimes n}) \cong (\CC^2)^{\otimes 2n}$. 
\end{prop}
\begin{proof}
First, we apply $\phi_1$ to $\End((\CC^2)^{\otimes n})$, obtaining $(\CC^2)^{\otimes n} \otimes ((\CC^2)^{\otimes n})^*$, which we apply $\phi_2$ to to get $(\CC^2)^{\otimes n} \otimes ((\CC^2)^*)^{\otimes n}$. Finally, we apply $\phi_3$ to each of the last $n$ components to obtain $(\CC^2)^{\otimes 2n}$. 
\end{proof}

\begin{cor}\label{bij}
There is an isomorphism $\End_\bU((\CC^2)^{\otimes n}) \cong ((\CC^2)^{\otimes 2n})^\bU$. 
\end{cor}
\begin{proof}
This statement follows from the fact that the above proposition is an isomorphism of $\bU$-modules. 
\end{proof}

Note that $\Ind_{\TL_n \otimes \TL_n}^{\TL_{2n}}\CC_\mathrm{triv}$ is embedded in $(\CC^2)^{\otimes 2n}$ and $\End_\bU((\CC^2)^{\otimes n})$ is isomorphic to $\TL_n$ by quantum Schur-Weyl duality (see \cite{jimbo}). Finally, we can describe the isomorphism in Corollary~\ref{bij} diagrammatically. 
\begin{thm} \label{embed}
The isomorphism described in Corollary~\ref{bij} can be described as an embedding $\TL_{n} \hookrightarrow \Ind_{\TL_n \otimes \TL_n}^{\TL_{2n}}\CC_\mathrm{triv}$ by mapping a diagram $D$ to $D$ composed horizontally with $n$ vertical lines composed vertically with the unique diagram with $n$ ``stacked'' links in the top line and $n$ quasi-simple links on the bottom line in $\Ind_{\TL_n \otimes \TL_n}^{\TL_{2n}}\CC_\mathrm{triv}$. 
\end{thm}
\begin{ex}
In the embedding described above, the procedure described in Theorem~\ref{embed} is shown in Figure~\ref{fig:cor1.10}. 
\begin{figure}[!ht]
\centering
\begin{tikzpicture}[scale=.25, baseline=(current bounding box.north)]
    \node at (1,9){$\bullet$};
    \node at (5,9){$\bullet$};
    \node at (9,9){$\bullet$};
    \node at (1,0){$\bullet$};
    \node at (5,0){$\bullet$};
    \node at (9,0){$\bullet$};
\draw (0,0) -- (10,0);
\draw (0,9) -- (10,9);
\draw (1,0) arc (180:0:2);
\draw (9,9) arc (360:180:2);
\draw (1,9) .. controls (1,1) and (9,8) .. (9,0);
\end{tikzpicture}
\hspace{1mm}
\begin{tikzpicture}[scale=.25, baseline=(current bounding box.north)]
\node at (0, 0){};
\node at (0, -4.5){$\rightarrow$};
\end{tikzpicture}
\hspace{1mm}
\begin{tikzpicture}[scale=.25, baseline=(current bounding box.north)]
\node at (1,9){$\bullet$};
\node at (5,9){$\bullet$};
\node at (9,9){$\bullet$};
\node at (13,9){$\bullet$};
\node at (17,9){$\bullet$};
\node at (21,9){$\bullet$};
\node at (1,0){$\bullet$};
\node at (5,0){$\bullet$};
\node at (9,0){$\bullet$};
\node at (13,0){$\bullet$};
\node at (17,0){$\bullet$};
\node at (21,0){$\bullet$};
\node at (1,18){$\bullet$};
\node at (5,18){$\bullet$};
\node at (9,18){$\bullet$};
\node at (13,18){$\bullet$};
\node at (17,18){$\bullet$};
\node at (21,18){$\bullet$};
\draw (0,0) -- (22,0);
\draw (0, 9)--(22,9);
\draw (0, 18)--(22,18);
\draw (13,9) arc (360:180:2);
\begin{scope}
\clip (0,0) rectangle (22,9);
\draw (11,9) ellipse (6 and 3);
\end{scope}
\begin{scope}
\clip (0,0) rectangle (22,9);
\draw (11,9) ellipse (10 and 4);
\end{scope}
\draw (9, 0) arc (180:0:2);
\begin{scope}
\clip (0,0) rectangle (22,9);
\draw (11,0) ellipse (6 and 3);
\end{scope}
\begin{scope}
\clip (0,0) rectangle (22,9);
\draw (11,0) ellipse (10 and 4);
\end{scope}
\draw (1,9) arc (180:0:2);
\draw (9,18) arc (360:180:2);
\draw (1,18) .. controls (1,10) and (9,17) .. (9,9);
\draw (13,9)--(13,18);
\draw (17,9)--(17,18);
\draw (21,9)--(21,18);
\end{tikzpicture}
\hspace{1mm}
\begin{tikzpicture}[scale=.25, baseline=(current bounding box.north)]
\node at (0, 0){};
\node at (0, -4.5){$\rightarrow$};
\end{tikzpicture}
\hspace{1mm}
\begin{tikzpicture}[scale=.25, baseline=(current bounding box.north)]
\node at (1,9){$\bullet$};
\node at (5,9){$\bullet$};
\node at (9,9){$\bullet$};
\node at (13,9){$\bullet$};
\node at (17,9){$\bullet$};
\node at (21,9){$\bullet$};
\node at (1,0){$\bullet$};
\node at (5,0){$\bullet$};
\node at (9,0){$\bullet$};
\node at (13,0){$\bullet$};
\node at (17,0){$\bullet$};
\node at (21,0){$\bullet$};
\draw (0,0) -- (22,0);
\draw (0, 9)--(22,9);
\draw (9,9) arc (360:180:2);
\begin{scope}
\clip (0,0) rectangle (22,9);
\draw (7,9) ellipse (6 and 4);
\end{scope}
\draw (21,9) arc (360:180:2);
\draw (9, 0) arc (180:0:2);
\begin{scope}
\clip (0,0) rectangle (22,9);
\draw (11,0) ellipse (6 and 4);
\end{scope}
\begin{scope}
\clip (0,0) rectangle (22,9);
\draw (11,0) ellipse (10 and 5);
\end{scope}
\end{tikzpicture}
\caption{An example of the embedding of $\TL_3$ into $\Ind_{\TL_3 \otimes \TL_3}^{\TL_6}\CC_\mathrm{triv}$. }
\label{fig:cor1.10}
\end{figure}
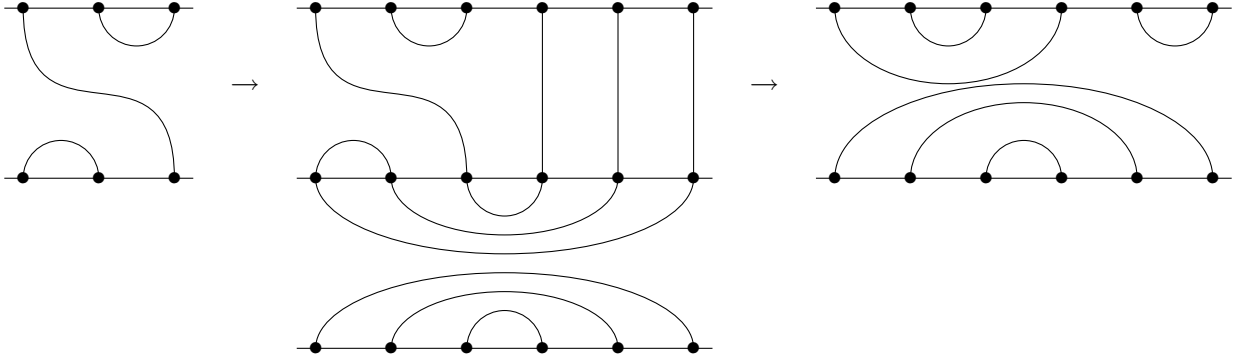
\end{ex}
Given a vector $v = v_{k_n} \otimes v_{k_{n-1}} \otimes \dots \otimes v_{k_1}$, where $k_i$ is $+$ or $-$ for all $i$, let $v^a$ denote $v_{-k_1} \otimes \dots \otimes v_{-k_{n-1}} \otimes v_{-k_n}$. A vector of the form $v^a \otimes v$ is called \emph{anti-symmetrical}. 
\begin{proof}[Proof of Theorem~\ref{embed}]
Let $D$ be a diagram in $\TL_n$ and let $D' = D_T \circ D_B$ be the corresponding diagram in $\Ind_{\TL_n \otimes \TL_n}^{\TL_n} \CC_\mathrm{triv}$ split into the composition as described. We prove that the action of $D'$ gives the same vector in $(\CC^2)^{\otimes 2n}$ as the identification of $D$ using Proposition~\ref{bij}. 

First, we describe the action of $D'$. The diagram $D_B$ acting on $\underbrace{v_- \otimes \dots \otimes v_-}_n \otimes \underbrace{v_+ \otimes \dots \otimes v_+}_n$ gives a sum of all anti-symmetrical vectors with $2n$ components. Furthermore, there is a factor of $-q^{-1}$ for each $v_-$ in the first $n$ components, or each $v_+$ in the last $n$ components. Finally, $D_T$ acts on $D_B \cdot \underbrace{v_- \otimes \dots \otimes v_-}_n \otimes \underbrace{v_+ \otimes \dots \otimes v_+}_n$ by replacing each antisymmetrical vector $v^a \otimes v$ with $(D \cdot v^a)\otimes v$. 

Now, we compute the embedding of $D$ into $(\CC^2)^{\otimes 2n}$ following the proof of Proposition~\ref{bij}. The embedding into $(\CC^2)^{\otimes n} \otimes ((\CC^2)^{\otimes n})^*$ is of the form $\sum_{i = 1}^{2^n} a_i (D \cdot v_i) \otimes v_i^*$, where $v_i$ is a standard basis vector of $(\CC^2)^{\otimes n}$, $v_i^*$ is taken with respect to the standard basis, and $a_i \in \CC$. Applying $\phi_2$ and $\phi_3$ replaces the last $n$ components with $v^a$ and creates a factor of $-q^{-1}$ for each $v_+$ in the last $n$ components. Note that this yields the same vector in $(\CC^2)^{\otimes n}$ as above.
\end{proof}

\section{The Dual Canonical and Canonical Basis Through the Hecke Algebra}\label{dualthroughhecke}
This section is a technical section that will be needed for the description of the canonical basis in Section~\ref{canonicalsec}. The results of this section are extracted from the work of Lusztig \cite{lusztig}. 
\subsection{Background for Canonical Basis}
Other than the dual canonical basis, $(\CC^2)^{\otimes n}$ has another important basis called the canonical basis, which satisfies the same upper triangularity property as the dual canonical basis and is self dual under another involution. The canonical basis element with leading term $x_{k_n} \otimes x_{k_{n-1}} \otimes \dots \otimes x_{k_1}$ is labelled by $x_{k_n} \dia x_{k_{n-1}} \dia \dots \dia x_{k_1}$. 

The dual canonical and canonical basis are related to the Hecke algebra by maps that send the spherical and aspherical basis to the dual canonical and canonical basis of $(\CC^2)^{\otimes n}$, respectively. 

In \cite{khovanov}, Khovanov defines the map $\widetilde{Seq}$ mapping $W^J$, the set of minimal coset elements of $W/W_J$, where $W_J$ is the Coxeter group generated by $S_k \times S_{n-k}$, to $(\CC^2)^{\otimes n}$ by sending $1$ to $\underbrace{v_- \otimes \dots \otimes v_-}_k \otimes \underbrace{v_+ \otimes \dots \otimes v_+}_{n-k}$ and having $s_i$ permute the $i$th and $(i+1)$th components in $\widetilde{Seq}(1)$. Furthermore, we abuse notation and define another map $\widetilde{Seq}: \mM \rightarrow (\CC^2)^{\otimes n}$ on the basis vectors by $\widetilde{Seq}(M_w) = \widetilde{Seq}(w)$. 

\begin{prop}[\cite{khovanov}] \label{sphericalcanonical}
The map $\widetilde{Seq}$ takes the canonical basis element $\underline M_w$ to the dual canonical basis element labeled by $\widetilde{Seq}(w)$. 
\end{prop}

Let $Seq$ be defined as $Seq\colon w \mapsto (-1)^{\ell(w)}\widetilde{Seq}(w)$. This map is similar to the map $Seq$ in \cite{khovanov}, but we define it differently for the purposes of this paper. 

\begin{prop}[\cite{khovanov}]
The map $Seq$ takes the canonical basis element $(-1)^{\ell(w)}\underline N_w$ to the canonical basis element labeled by $\widetilde{Seq}(w)$. 
\end{prop}

Finally, we describe the most important property of the canonical basis, its duality to the dual canonical basis. 

\begin{thm}[\cite{khovanov}]\label{khovanovdual}
If we define a pairing in $(\CC^2)^{\otimes n}$ by \[\langle x_{k_n} \otimes x_{k_{n-1}} \otimes \dots \otimes x_{k_1}, x_{k_1'}\otimes \dots \otimes x_{k_{n-1}'} \otimes x_{k_n'} \rangle = \delta_{k_1, k_1'}\dots \delta_{k_{n-1}, k_{n-1}'} \delta_{k_n, k_n'},\] we also have that \[\langle x_{k_n} \heart x_{k_{n-1}} \heart \dots \heart x_{k_1}, x_{k_1'}\dia \dots \dia x_{k_{n-1}'} \dia x_{k_n'} \rangle = \delta_{k_1, k_1'}\dots \delta_{k_{n-1}, k_{n-1}'} \delta_{k_n, k_n'}.\]
\end{thm}

In Khovanov's proof, he doesn't rely on the Hecke algebra, instead proving this duality directly in $(\CC^2)^{\otimes n}$. We prove this property by proving the analog for the canonical basis in the spherical and aspherical modules, which also give results about the computation of the canonical basis in $\mM^*$ and $\mN^*$, the dual spaces to the spherical and aspherical modules, that are independent from the spin representation. 
\subsection{Duality in the Hecke Algebra}
In this section, we prove Theorem~\ref{flipMN}, and show that it implies the duality between the dual canonical and canonical bases (Theorem~\ref{khovanovdual}). 

In order to prove facts about the canonical basis of $\mM^*$ and $\mN^*$, we must first give a relation between the canonical basis of $\mH$ and $\mH^*$, the dual space to the Hecke algebra. Let $w_0$ denote the longest element of $W$ and $w_{0, J}$ denote the longest element in $W_J$. Define $\flip\colon \mH \rightarrow \mH$ as $\flip\colon H_w \mapsto H_{w_0w}$, extended linearly. Let $w_f$ be the longest element in $W^J$. Define $\flip\colon \mN \rightarrow \mN$ and $\flip\colon \mM \rightarrow \mM$ by $\flip(M_w) = M_{w_fw}$ and $\flip(N_w) = N_{w_fw}$. 

\begin{thm} \label{flipMN}
We show that $\langle \flip(\underline M_{w_fw}), \underline N_{x} \rangle = (-1)^{\ell(x)}\delta_{x, w}$, with the pairing being defined as $\langle M_{w}, N_x \rangle = (-1)^{\ell(x)}\delta_{w, x}$.  
\end{thm}

We can show that our pairing on the Hecke algebra proves the duality between the canonical basis and the dual canonical basis. 

\begin{prop}
Theorem~\ref{flipMN} implies Theorem~\ref{khovanovdual}. 
\end{prop}
\begin{proof}
The fact that $\underline M_{w_fw}$ is paired with $N_x$ is due to the fact that the pairing of the canonical and dual canonical basis in $(\CC^2)^{\otimes n}$ is defined ``reversed,'' or that $x_{k_n} \heart x_{k_{n-1}} \heart \dots \heart x_{k_1}$ is paired with $x_{k_1'}\dia \dots \dia x_{k_{n-1}'} \dia x_{k_n'}$. The fact that we flip $\underline M_{w_fw}$ is due to the fact that the pairing in $(\CC^2)^{\otimes n}$ is defined ``reversed.'' Furthermore, the fact that there is a factor of $(-1)^{\ell x}$ in both the pairing and the result is due to the fact that $Seq$ is defined with a power of $(-1)^{\ell(w)}$, and furthermore, it maps $(-1)^{\ell(w)}\underline N_w$ to the canonical basis element. 
\end{proof}

In order to prove Theorem~\ref{flipMN}, we must first prove some other results involving the $\flip$ function. 

\begin{prop}\label{flipinH}
Let $S_w$ denote the standard basis elements of $\mH^*$, where $S_w$ is defined by $\langle S_w, H_x \rangle = (-1)^{\ell(w)} \delta_{w, x}$. We show that $\flip(B_{w_0w})$ is equal to the canonical basis element of $\mH^*$, which we denote $D_w$, under the identification $H_w \mapsto S_w$. 
\end{prop}
\begin{proof}
By definition, $B_{w_0w} = H_{w_0w} + \sum_{v \prec w_0w} p_{v, w_0w}H_v$. We can modify what the sum is over, to obtain $B_{w_0w} = H_{w_0w} + \sum_{w_0v \prec w_0w} p_{w_0v, w_0w}H_{w_0v}$. Finally, 
\[ \flip(B_{w_0w}) = H_w + \sum_{w_0v \prec w_0w} p_{w_0v, w_0w}H_{v}. \]

Consider $\phi\colon \mH \rightarrow \mH^*$ by $\phi: H_w \mapsto S_w$. We have that 
\[ \phi(\flip(B_{w_0w})) = S_w + \sum_{w_0v \prec w_0w} p_{w_0v, w_0w}S_{v}. \]

Proposition 11.4 of \cite{lusztig} states that $q_{v, w} = p_{w_0v, w_0w}$, where $q_{w, v}$ are the Kazhdan-Lusztig polynomials in $\mH^*$. Furthermore, given $w$, the $v$ such that $w_0v \prec w_0w$ are precisely the $v$ such that $v \succ w$. Therefore, 
\begin{align*}
\phi(\flip(B_{w_0w})) &= S_w + \sum_{w_0v \prec w_0w} p_{w_0v, w_0w}S_{v}\\
&= S_w + \sum_{v \succ w} q_{w, v}S_{v}\\
&= D_w. \qedhere
\end{align*}
\end{proof}

\begin{cor}
Under the pairing of elements in $\mH$ by $\langle H_x, H_y \rangle = (-1)^{\ell(x)}\delta_{x, y}$, we have that $\langle \flip(B_{w_0w}), B_x\rangle = (-1)^{\ell(x)}\delta_{x, w}$. 
\end{cor}

\begin{proof}
Tanisaki \cite{tanisaki} shows that $\langle D_w, B_x \rangle = (-1)^{\ell(w)} \delta_{w, x}$ under the pairing $\langle S_x, H_w \rangle = (-1)^{\ell(x)} \delta_{x, y}$. This implies the corollary. 
\end{proof}

We also prove some lemmas about spherical and aspherical modules. 
\begin{lem}[\cite{losev}] \label{iota}
The map $\iota\colon \mN \hookrightarrow \mH$ by $\iota\colon N_w \mapsto B_{w_{0, J}}H_w = (\sum_{w \in W_J} (-q)^{-\ell(w_{0, J})+\ell(w)}H_w)H_w$ sends $\underline N_w$ to $B_{w_{0, J}w}$ for all $w \in W^J$. 
\end{lem}
\begin{lem} \label{wequality}
We have the equality $w_{0, J}w_f = w_0$. 
\end{lem}
\begin{proof}
This statement follows from the bijection $W_J \times W^J \rightarrow W$ by $(x, y) \mapsto xy$; as $\ell(xy) = \ell(x) + \ell(y)$ for $x \in W_J$, $y \in W^J$, $w_0$ must be the product of the longest element of $W_J$ and the longest element of $W^J$. 
\end{proof}

Define a surjection $\sigma: \mH \rightarrow \mM$ as follows. Define $\sigma_1\colon \mH \rightarrow \mH$ by $\sigma_1: H_w \mapsto q^{-\ell(w') + \ell(w)} H_{w'}$, where $w' \in W^J$ is the element of shortest length such that $w$ and $w'$ are in the same $W_J$-coset. Define $\sigma_2\colon \mH \rightarrow \mM$ by mapping $H_w$ to $M_{w''}$, where $w'' \in W^J$ is the shortest length element in the $W_J$-coset that $w$ is in. Define $\sigma = \sigma_2 \circ \sigma_1$. 
\begin{lem}[\cite{losev}]
The surjection $\sigma: \mH \rightarrow \mM$ maps $B_w$ to $\underline M_w$ for $w \in W_J$. 
\end{lem}

Similarly, we can define the surjection $\sigma^*: \mH^* \rightarrow \mN^*$ as follows. Let $Q$ denote the standard basis for $\mN$. Given an element $S_w$ in $\mH^*$, let $w'$ be the element in $w_{0, J}W^J$ such that $w$ and $w'$ are in the same $W_J$-coset (so $w'$ is the longest element in the coset $w$ is in). Then, define $\sigma_1^*\colon \mH^* \rightarrow \mH^*$ such that $\sigma^*_1: S_w \mapsto q^{-\ell(w') + \ell(w)} S_{w'}$. Furthermore, define $\sigma_2^*\colon \mH^* \rightarrow \mN^*$ mapping $S_w$ to $Q_{w''}$, where $w'' \in W^J$ is the shortest length element in the $W_J$-coset that $w$ is in. Define $\sigma^* = \sigma_2^* \circ \sigma_1^*$. 

\begin{lem} \label{dualaspherical}
The surjection $\sigma^*\colon \mH^* \rightarrow \mN^*$ maps $D_{w_{0, J}w}$ to $\underline Q_w$ for $w \in W_J$. 
\end{lem}
\begin{proof}
First of all, $\sigma^*$ preserves upper triangularity, as the term $S_{w_{0, J}w}+\sum_{x \succ w}q^{-1}\ZZ[q^{-1}]S_{w_{0, J}x}$ maps to $Q_w + \sum_{x \succ w}q^{-1}\ZZ[q^{-1}]Q_x$. 

Now, we prove that $\sigma^*$ commutes with the bar involution. By duality, $\sigma^*$ commutes with the bar involution if and only if $(\sigma^*)^*\colon \mN \rightarrow \mH$ defined by $\langle (\sigma^*)^*(N_w), S_v \rangle = \langle N_w, \sigma^*(S_v) \rangle$ for all $w \in W_J$ and $v \in W$ commutes with the bar involution. 

Let $(\sigma^*)^*(N_w) = \sum_{v \in W}a_vT_v$. We have that 
\[\langle (\sigma^*)^*(N_w), S_v \rangle = \langle N_w, \sigma^*(S_v) \rangle = \langle N_w, q^{-\ell(v')+\ell(v)}Q_{v''}\rangle\] where $v'$ and $v''$ are the longest and shortest elements of the $W_J$-coset that $w$ is in, respectively. This is nonzero when $v$ is in the same $W_J$-coset as $w$. In that case, $\langle N_w, q^{-\ell(v')+\ell(v)}Q_{v''}\rangle = (-1)^{\ell(w)}q^{-\ell(v')+\ell(v)}$ so $a_v = (-1)^{\ell(v)}(-1)^{\ell(w)}q^{-\ell(v')+\ell(v)}$. 

Therefore, we have that 
\begin{align*}
(\sigma^*)^*(N_w) &= \sum_{v \in W_Jw}(-1)^{\ell(v)}(-1)^{\ell(w)}q^{-\ell(v')+\ell(v)}T_v \\
&= \sum_{x \in W_J} (-1)^{\ell(xw)}(-1)^{\ell(w)}q^{-\ell(xw_{0, J})+\ell(xw)}T_{xw} \\
&= \sum_{x \in W_J}(-1)^{\ell(x)}q^{-\ell(w_{0, J})+\ell(x)}T_{xw} \\
&= (-1)^{\ell(w_{0, J})}\sum_{x \in W_J}(-q)^{-\ell(w_{0, J})+\ell(x)}T_{xw}.
\end{align*}
From Lemma~\ref{iota}, $(\sigma^*)^* = (-1)^{\ell(w_{0, J}))}\iota$, and as $(-1)^{\ell(w_{0, J}))}$ is a constant and $\iota$ commutes with the bar involution (see \cite{losev}), $(\sigma^*)^*$ and hence $\sigma^*$ does so too. 

As $\sigma^*(D_{w_{0, J}w})$ satisfies both upper triangularity and self dual conditions, it must map to $\underline Q_w$. 
\end{proof}

As before, define $\phi\colon \mH \rightarrow \mH^*$ by $\phi\colon H_w \mapsto S_w$. Furthermore, define $\psi: \mM \rightarrow \mN^*$ by $M_w \mapsto Q_w$. Finally, define $\sf\colon \mH \rightarrow \mM$ as follows. Define $\sf_1 = \phi^{-1} \circ {\sigma_1}^* \circ \phi$ and $\sf_2 = \psi^{-1}\circ {\sigma_2}^* \circ \psi$. Define $\sf = \sf_2 \circ \sf_1$. 

\begin{lem}\label{nondualaspherical}
We have the equality $\flip(\underline M_w) = \sf(\flip(B_{w}))$ for all $w \in W^J$. 
\end{lem}
\begin{proof}
We decompose $\sf = \sf_2 \circ \sf_1$, so we attempt to prove that $\flip(\underline M_w) = \sf_2(\sf_1(\flip(B_{w})))$. 

We know that $B_w = H_{w} + \sum_{v \prec w} p_{v, w}H_v$, and $\underline M_w = \sigma(B_w)$. Therefore, it is sufficient to show that 
\begin{equation}\label{nondualeq}
\flip(\sigma(H_w)) = \sf_2(\sf_1(H_{w_0w})). 
\end{equation}
Let $w''$ be the element with smallest length in the $W_J$-coset of $w$. Then, \[\flip(\sigma(H_w)) = \flip(q^{\ell(w'')-\ell(w)}M_{w''}) = q^{\ell(w'')-\ell(w)}M_{w_fw''}.\]

If $w''$ is the element with smallest length in the $W_J$-coset of $w$, the longest element in the $W_J$-coset of $w_0w$ will be $w_0w''$. We have that $\sf_2(\sf_1(H_{w_0w})) = \sf_2(q^{-\ell(w_0w'')+\ell(w_0w)}S_{w_0w''})$. By Lemma~\ref{wequality}, this is equal to $\sf_2(q^{-\ell(w_0w'')+\ell(w_0w)}S_{w_{0, J}w_fw''}) = q^{-\ell(w_0w'')+\ell(w_0w)}M_{w_fw''}$. 

We have that $-\ell(w_0w'')+\ell(w_0w) = -(\ell(w_0) - \ell(w'')) + \ell(w_0) - \ell(w) = \ell(w'') - \ell(w)$. Therefore, the two sides of Equation~\ref{nondualeq} are equal. 
\end{proof}

\begin{proof}[Proof of Theorem~\ref{flipMN}]
We have that $\sigma^* \circ \phi = \psi \circ \sf$, which combined with $\phi(\flip(B_w)) = D_{w_0w}$ (see Proposition~\ref{flipinH}) implies that $\sigma^*(\phi(\flip(B_w))) = \sigma^*(D_{w_0w})$, which gives that $\psi(\sf(\flip(B_w)) = \sigma^*(D_{w_{0, J}w})$. Since $w \in W^J$, we can substitute $w = w_fx$ for some $x \in W_J$ to obtain $\psi(\sf(\flip(B_{w_fx})) = \sigma^*(D_{w_0w_fx}) = \sigma^*(D_{w_{0, J}x})$ where the last equality is from Lemma~\ref{wequality}. From Lemmas~\ref{dualaspherical} and~\ref{nondualaspherical}, we get that $\psi(\flip(\underline M_{w_fx})) = \underline Q_x$. 

As $\langle \underline N_w, \underline Q_x \rangle = (-1)^{\ell(x)}\delta_{w, x}$ given the pairing $\langle N_w, Q_x \rangle = (-1)^{\ell(x)} \delta_{w, x}$ due to the definition of the dual aspherical basis, the theorem is implied. 
\end{proof}
\subsection{Corollaries in Dual Spaces to Spherical and Aspherical Modules}
We obtain some corollaries in the dual spaces of spherical and aspherical modules as a byproduct of our results. In the proof of Theorem~\ref{flipMN}, we obtain the following corollary. 

\begin{cor}
We have that $\underline Q_w$ is equal to $\flip(\underline M_{w_fw})$ under the identification $M_w \mapsto Q_w$. 
\end{cor}

Let $R_w$ be the standard basis elements and let $\underline R_w$ be the canonical basis elements of $\mM^*$. As $(-1)^{\ell(x)}\delta_{x, w} = \langle \flip(\underline M_{w_fw}), \underline N_{x} \rangle = \langle \flip(\underline M_{w}), \underline N_{w_fx} \rangle = \langle \underline M_{w}, \flip(\underline N_{w_fx}) \rangle$, we also obtain the following corollary. 

\begin{cor}
We have that $\underline R_w$ is equal to $\flip(\underline N_{w_fw})$ under the identification $N_w \mapsto R_w$. 
\end{cor}

With these results, we can compute the canonical bases in $\mH^*$, $\mN^*$, and $\mM^*$ from the canonical bases in $\mH$, which is an interesting result, even independent from the spin representation. 

\begin{rem}
Our approach allows us to work with the spin representation through the Hecke algebra, which is more convenient as the bar involution in the Hecke algebra is simpler than the involution in the spin representation. 
\end{rem}

\section{Alternative Axiomatic Definition of the Canonical Basis}\label{canonicalsec}

In this section, we describe an alternative axiomatic definition of the canonical basis, leading to a more explicit computation of it. 

\begin{lem}
We can define a right action of $\TL_n$ on $((\CC^2)^{\otimes n})$ by defining the map $\zeta^*\colon \TL_n \rightarrow \End((\CC^2)^{\otimes n})$ such that 

\[ \zeta^*(e_i) = \displaystyle \begin{pmatrix}
0 & 0 & 0 & 0\\
0 & -q^{-1} & 1 & 0\\
0 & 1 & -q & 0\\
0 & 0 & 0 & 0
\end{pmatrix}_{n-i, n-i+1}. \] 
\end{lem}
\begin{proof}
We can verify that all the relations hold. 
\end{proof}

With this action, we can describe a new axiomatic definition of the canonical basis. Given a canonical basis element $v_{k_n} \dia \dots \dia v_{k_1}$ (see \cite{khovanov}), let $d(v_{k_n} \dia \dots \dia v_{k_1})$ denote the diagram associated with the label of it using the procedure described in Section~\ref{diagrambij}. Furthermore, if diagram $D$ corresponds to canonical basis element $v_{k_n} \heart \dots \heart v_{k_1}$ by the bijection described in Section~\ref{diagrambij}, then define $\flip(D)$ to be the diagram corresponding to $v_{k_1} \heart \dots \heart v_{k_n}$. 

\begin{thm}\label{canonicalthm}
Given a canonical basis element $v_{k_n} \dia \dots \dia v_{k_1}$, let $D_1$ denote the diagram $d(v_{k_n} \dia \dots \dia v_{k_1})$. The canonical basis element $v_{k_n} \dia \dots \dia v_{k_1}$ can be defined as the unique element $w \in (\CC^2)^{\otimes n}$ such that under the pairing defined in Theorem~\ref{khovanovdual}, \[\langle v_- \otimes \dots \otimes v_- \otimes v_+ \otimes \dots \otimes v_+, \zeta^*(D_2)w \rangle = \delta_{D_2, \flip(D_1)}\]
for all possible $D_2$. 
\end{thm}
\begin{proof}
First, we note that $\zeta^*$ is constructed to be the unique map such that $\langle D \cdot v, w \rangle = \langle v, \zeta^*(D)w \rangle$ for $v$, $w \in (\CC^2)^{\otimes n}$. By Theorem~\ref{khovanovdual}, we know that \[\langle v_{k_n} \heart \dots \heart v_{k_1}, v_{k_n} \dia \dots \dia v_{k_1} \rangle = \delta_{d(v_{k_n} \heart \dots \heart v_{k_1}), \flip(d(v_{k_n} \dia \dots \dia v_{k_1}))}.\] Let $D_2$ be equal to $d(v_{k_n} \heart \dots \heart v_{k_1})$. 
By Theorem~\ref{diagramdualcomp}, we obtain that 
\begin{align*}
\langle v_{k_n} \heart \dots \heart v_{k_1}, v_{k_n} \dia \dots \dia v_{k_1} \rangle &= \langle D_2 \cdot v_- \otimes \dots \otimes v_- \otimes v_+ \otimes \dots \otimes v_+, v_{k_n} \dia \dots \dia v_{k_1} \rangle\\
&= \langle v_- \otimes \dots \otimes v_- \otimes v_+ \otimes \dots \otimes v_+, \zeta^*(D_2)(v_{k_n} \dia \dots \dia v_{k_1}) \rangle.\\
\end{align*}
Therefore, $v_{k_n} \dia \dots \dia v_{k_1}$ is the unique element in $(\CC^2)^{\otimes n}$ satisfying the pairing condition described in the theorem. 
\end{proof}
%
%
%
%
%
%

\section*{Acknowledgements}
The author would like to thank their mentor Dr.~Vasily Krylov for his knowledge, guidance, and enthusiasm in answering questions. Furthermore, the author would like to thank the PRIMES-USA program and its director Dr.~Slava Gerovitch for giving them this incredible opportunity to learn and conduct research. Finally, the author would like to thank Dr.~Tanya Khovanova and Ivan Motorin for helpful feedback on the paper. 
%

\bibliographystyle{plain}
\bibliography{refs}
\end{document}